\documentclass[11pt]{amsart}

\usepackage{amsmath,amssymb,amsfonts,amsthm,enumerate}
\usepackage{vmargin}
\usepackage[dvips]{graphicx}
\usepackage{color}
\input{amssym.def}
\input{amssym}
\input{epsf}
\usepackage{supertabular}
\usepackage{array}
\usepackage{multicol}

\usepackage[titletoc]{appendix}

\def\R{\mathbb R}
\def\N{\mathbb N}

\def\E{\cal E}

\def\cal{\mathcal}

\renewcommand{\bar}{\overline}
\newcommand{\vc}[1]{{\color{red}#1}}

\def\H{\cal H}

\def\In{\mathcal{I}}
\def\Out{\mathcal{O}}

\def\<{\langle}
\def\>{\rangle}

\def\begeq{\begin{equation}}
\def\endeq{\end{equation}}
\def\begar{\begin{eqnarray}}
\def\endar{\end{eqnarray}}
\def\begar*{\begin{eqnarray*}}
\def\endar*{\end{eqnarray*}}
\def\begal{\begin{align}}
\def\endal{\end{align}}
\def\begal*{\begin{align*}}
\def\endal*{\end{align*}}

\newtheorem{Thm}{Theorem}

\newtheorem{Lem}[Thm]{Lemma}

\newtheorem{Cor}[Thm]{Corollary}
\newtheorem{Prop}[Thm]{Proposition}

\newtheorem{Problem}[Thm]{Problem}

\newtheorem{Claim}{Claim}

\numberwithin{equation}{section}
\numberwithin{Thm}{section}

\theoremstyle{definition}
\newtheorem{Def}[Thm]{Definition}
\newtheorem{Rk}[Thm]{Remark}

\theoremstyle{remark}

\newtheorem*{Thm*}{Theorem}
\newtheorem*{Lem*}{Lemma}
\newtheorem*{Conj*}{Conjecture}
\newtheorem*{Cor*}{Corollary}
\newtheorem*{Def*}{Definition}
\newtheorem*{Prop*}{Proposition}
\newtheorem*{Exo*}{Exercise}
\newtheorem*{Exs*}{Examples}
\newtheorem*{Ex*}{Example}
\newtheorem*{Rk*}{Remark}
\newtheorem*{Rks*}{Remarks}

\author{Vincent Calvez}
\address{Unit\'e de Math\'ematiques Pures et
Appliqu\'ees, Ecole Normale Sup\'erieure de Lyon, CNRS UMR 5669, and
 project-team Inria NUMED, Lyon, France. E-mail: {\tt
 vincent.calvez@ens-lyon.fr}}
 
\author{Thomas Gallou\"et}
\address{Project-team MEPHYSTO, Inria Lille - Nord Europe, Villeneuve d'Ascq, France. E-mail: {\tt
 thomas.gallouet@inria.fr}}

\begin{document}


\title[Particle approximation of the 1D Keller-Segel equation]{Particle approximation of the one dimensional Keller-Segel equation, stability and rigidity of the blow-up}


\begin{abstract}
We investigate a particle system which is a discrete and deterministic approximation of the one-dimensional Keller-Segel equation with a logarithmic potential. The particle system is derived from the gradient flow of the homogeneous free energy written in Lagrangian coordinates. We focus on the description of the blow-up of the particle system, namely: the number of particles involved in the first aggregate, and the limiting profile of the rescaled system. We exhibit basins of stability for which the number of particles is critical, and we prove a weak rigidity result concerning the rescaled dynamics. This work is complemented with a detailed analysis of the case where only three particles interact.
\end{abstract}

\maketitle
\section{Introduction and main results}

We investigate the numerical analysis of a deterministic particle approximation of the Keller-Segel equation that was introduced in \cite{BCC08}. We focus on the blow-up issue at the discrete level, when a cloud of particles merge together to form the first singular aggregate. We restrict to a one-dimensional version of the Keller-Segel equation which shares common features with the classical two-dimensional problem.

We take advantage of the one-dimensional structure to design a numerical particle scheme which possesses the same geometrical structure as the continuous problem. The strategy is as follows: it is well known that continuous systems of diffusive self-interacting particles possess a gradient flow structure with respect to the free energy \cite{AGSbook,JKO, otto01, oldandnew}. The euclidean distance between particles translates into the Wasserstein distance between distributions of particles in the space of probability measures. We proceed the other way around: we discretize the free energy in Lagrangian coordinates \cite{GosTos06}, then we consider the time continuous gradient flow with respect to the euclidean metric.

The free energy of the one-dimensional Keller-Segel equation with a logarithmic interaction kernel reads as follows in the Lagrangian coordinates:
\begin{equation} \label{eq:energy lagrangian}
E(X) = - \int_{(0,1)}  \log \left(\dfrac {dX}{dm}\right)\, dm + \chi \iint_{(0,1)^2} \log |X(m)- X(m')|\, dmdm'\, .
\end{equation}
Here, $X: (0,1)\to \R$ encodes the position of particles  with respect to the partial mass $m\in (0,1)$. We assume that $X$ belongs to the energy space $\mathcal X$:
\[\mathcal X = \left\{X\in L^2(0,1)\cap\mathcal C^1(0,1) \, : \; \dfrac {dX}{dm}\geq 0\, , \; \text{and}\;  E(X)\; \text{is finite}\right\}\, .\]
The first contribution in \eqref{eq:energy lagrangian} is the internal energy, which accounts for the diffusion of particles. The second contribution is the interaction energy which accounts for the self-attraction of particles.

It is worth noticing that the blow-up phenomenon in the supercritical case can be simply deduced from the logarithmic homogeneity of \eqref{eq:energy lagrangian} with respect to dilations $X_\lambda = \lambda X$. On the other hand, the global existence in the subcritical case is also a consequence of the energy structure \cite{BDP06}. Therefore if we discretize the problem in such a way to keep those two properties (logarithmic homogeneity and gradient flow structure) then we can ensure the critical mass phenomenon at the discrete level too.

The numerical scheme is thus designed as follows: (i) we discretize the density of particles into a finite number of particles, $N$ fixed throughout the paper, having equal masses $h_N = \frac1{N+1}$, at positions $(X_i)_{1\leq i\leq N}$ such that
$X_1<X_2<\dots<X_N$; (ii) we opt for a simple  discretization of \eqref{eq:energy lagrangian}:
\begin{equation}\label{EnergieKelSelDiscret}
\E\left(X\right)=-\sum_{i=1}^{N-1} \log\left(X_{i+1}-X_i\right) +\chi h_N\sum_{1\le i\neq j\le N}  \log \vert X_i-X_j \vert\,,
\end{equation}
(notice that we have omitted a factor $h_N$ in front of both contributions of \eqref{EnergieKelSelDiscret} for the sake of clarity);
(iii) we take the finite-dimensional euclidean gradient flow of $\E_N$.
This gives,\begin{equation}\label{flotgradientdiscretexplicite}
\dot X_i = -\dfrac{1}{X_{i+1}-X_i} +   \dfrac{1}{X_{i}-X_{i-1}} + \quad 2\chi
h_N \sum_{j \neq i } \dfrac{1}{X_j-X_i}\,,
\end{equation}
complemented with the dynamics of the extremal points
\begin{equation}\label{flotgradientdiscretexplicite-boundary}
\left\{\begin{array}{ccl}\dot X_1 =&-\dfrac{1}{X_{2}-X_1} &+ \quad 2\chi h_N\displaystyle\sum_{j \neq 1}
\dfrac{1}{X_j-X_1}\medskip\\
\dot X_N=&\dfrac{1}{X_{N}-X_{N-1} }&+ \quad 2\chi h_N \displaystyle\sum_{j \neq N}
\frac{1}{X_i-X_N}
\end{array}\right.
\end{equation}

The particle scheme \eqref{flotgradientdiscretexplicite}--\eqref{flotgradientdiscretexplicite-boundary} presents several advantages. First it is very similar to the two-dimensional Keller-Segel equation from the geometric viewpoint ({\em i.e.} the gradient flow of a homogeneous functional). Thus it captures accurately the critical mass phenomenon. Second the lagrangian viewpoint avoids truncature of a spatial domain, which is usually the case for finite volume schemes, see {\em e.g.} \cite{Fil06}. On the other hand we assume that extremal particles are not interacting with $\pm \infty$ \eqref{flotgradientdiscretexplicite-boundary}. Third, as blow-up occurs the scheme \eqref{flotgradientdiscretexplicite} dynamically adapt the "mesh" (from the eulerian viewpoint) to increase accuracy at the blow-up point since many particles converge towards it.

A stochastic particle approximation of the two-dimensional Keller-Segel equation has been extensively studied by
Ha\v{s}kovec and Schmeiser in a couple of papers \cite{HasSch09,schhas12}. The first paper is concerned with the design a numerical scheme which enables to follow {\em heavy} aggregates after the occurrence of blow-up in the spirit of \cite{Dschm09,LuSuVel12,Velazquez04a,Velazquez04b} (see also \cite[Chapter 7]{theseannedevys} for a similar work for \eqref{flotgradientdiscretexplicite}--\eqref{flotgradientdiscretexplicite-boundary}).
The second paper analyses the limit of a large number of particles. The author investigate the Boltzmann hierarchy obtained in the limit $N\to +\infty$, and prove its compatibility with the  measure-valued solutions {\em \`a la Poupaud} defined in \cite{Dschm09}. In addition they focus on the case of two interacting  particles only.
The main difference between \cite{HasSch09,schhas12} and our approach relies on the treatment of the diffusion term. Our approach is fully deterministic and transcripts the diffusion of particles into a pressure term that pushes apart neighbouring particles, as can be seen in \eqref{flotgradientdiscretexplicite}--\eqref{flotgradientdiscretexplicite-boundary}. We also refer to \cite{pascomprisdemander} for a deterministic approximation of the two-dimensional Keller-Segel equation in Lagrangian coordinates.

A challenging question in the analysis of the two-dimensional Keller-Segel equation consists in proving that the first blow-up set contains exactly the critical amount of mass. This question requires to understand very precisely the dynamics close to the blow-up time/point. This question was raised in \cite{ChPer81}. A constructive partial answer was given in \cite{HerVel97} in the radial case using formal matching asymptotics. It has been extensively studied in \cite{Suzukibook}. In \cite{KavSou08} the authors investigate the critical radially symmetric case, for which it is known that blow-up occurs in infinite time \cite{BCM08}. In a recent work \cite{Raphael-Schweyer}, the authors rigorously derive the blow-up dynamics obtained in \cite{HerVel97}. By using  very powerful techniques of critical blow-up problems developped for parabolic and dispersive equations, they are able to characterize very precisely the dynamics of blow-up close to the critical ground state $Q(x) = \frac{8}{(1 + |x|^2)^2}$. They prove that for initial data having supercritical mass and close to the ground state in some weighted $H^2$ norm, the solution blows-up with a universal blow-up rate and a universal profile given by a dilation of the ground state.

In the present work we address similar questions for the discrete problem \eqref{flotgradientdiscretexplicite}--\eqref{flotgradientdiscretexplicite-boundary}:
\begin{enumerate}[(i)]
\item prove that the blow-up of a critical amount of mass (here a critical number of particles) is a stable process,
\item investigate the dynamics close to the blow-up time/point in the stable regime.
\end{enumerate}
We relax several difficulties specific to the continuous setting. As a drawback we miss refined dynamics such as the logarithmic correction of the blow-up rate \cite{HerVel97, Raphael-Schweyer}. On the other hand our analysis does not rely on any perturbation analysis. Alternatively, the Lagrangian formulation is well suited to separate inner and outer contributions to the blow-up, as explained below. In addition we completely describe the case of three interacting particles. There it is clear that blowing-up with the critical amount of mass is a generic process. Only very peculiar symmetric cases break this structure.


\begin{figure}
\includegraphics[width=0.48\linewidth]{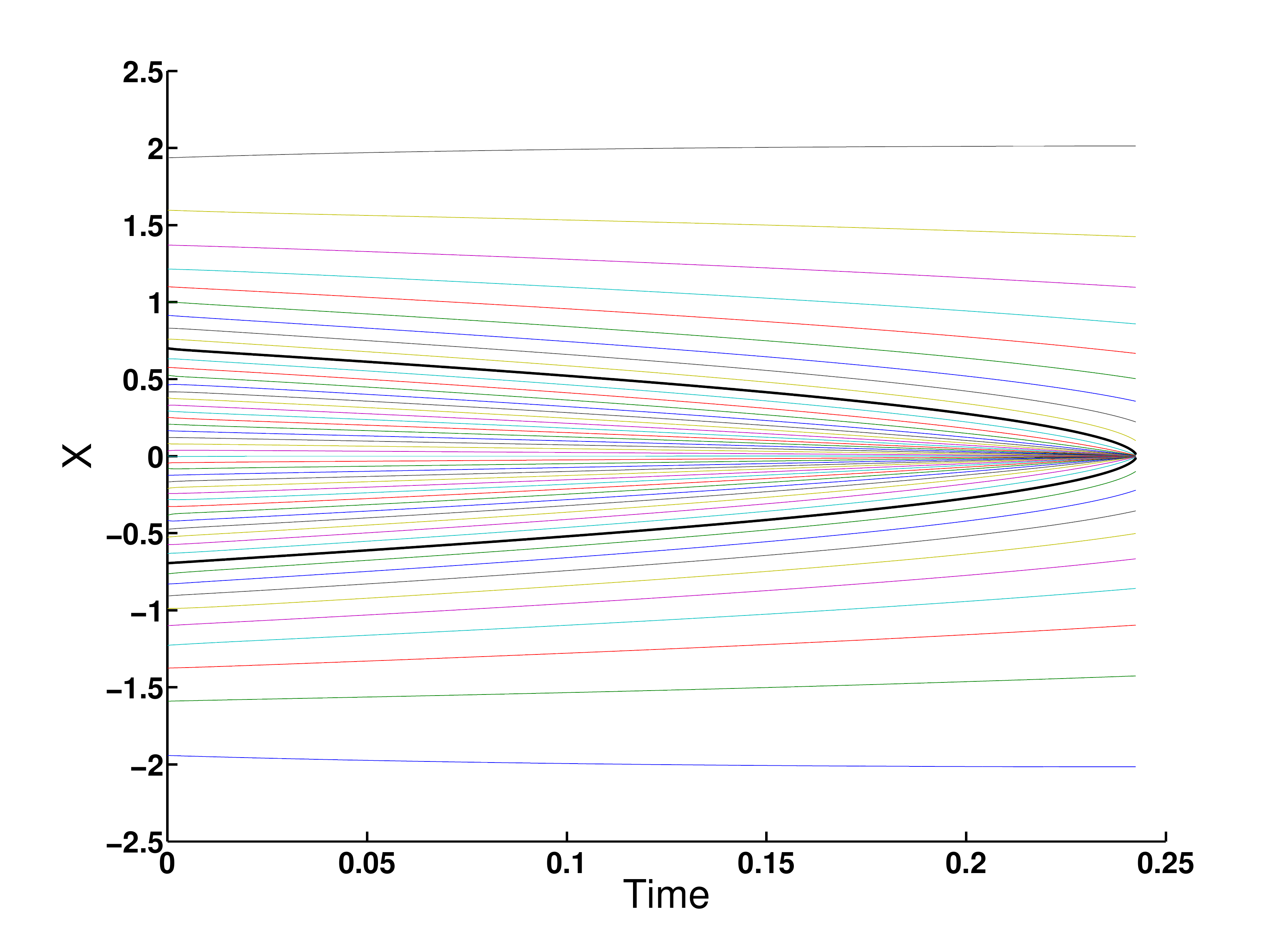}\quad
\includegraphics[width=0.48\linewidth]{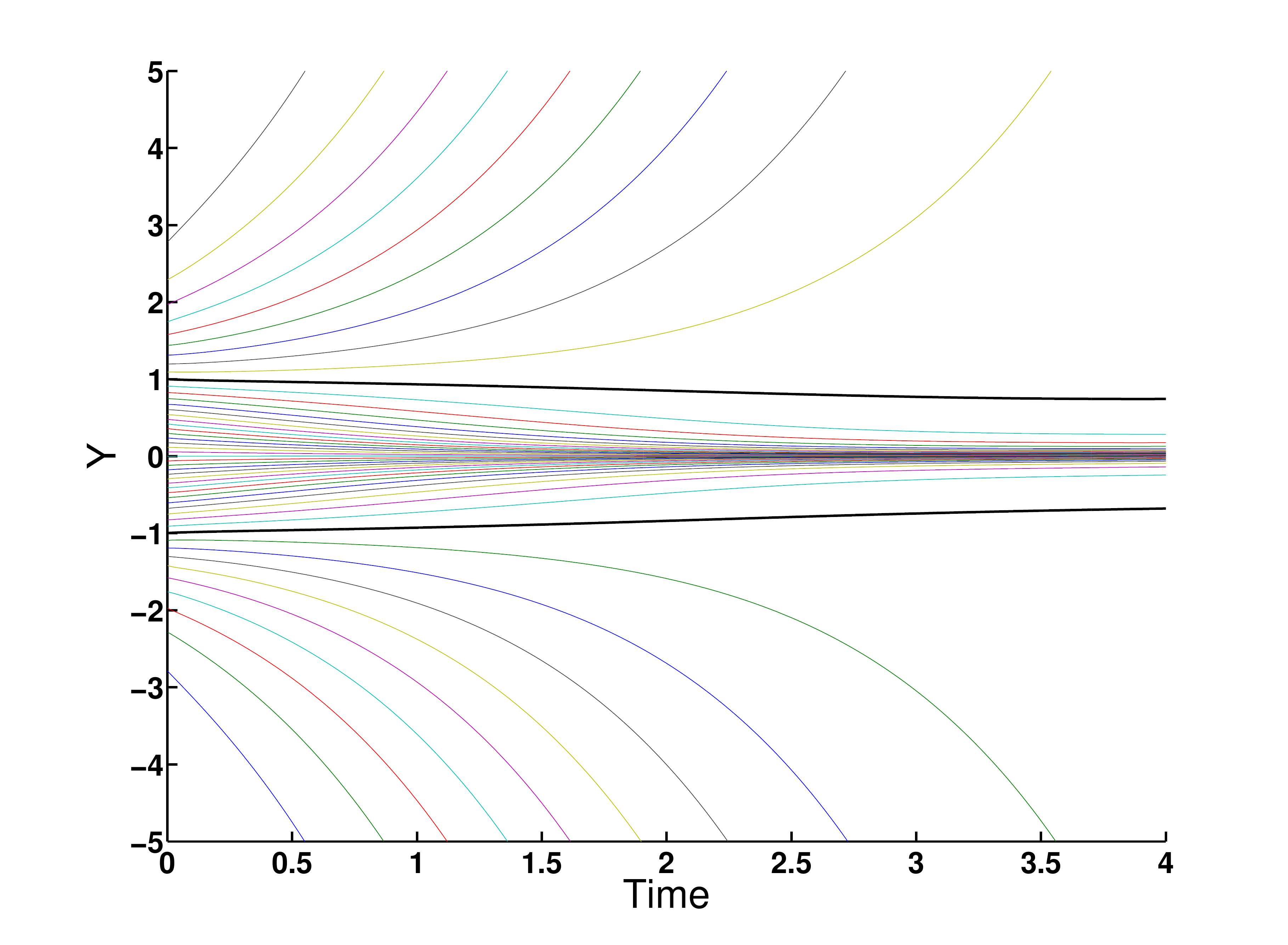}
\caption{Numerical simulations of the discrete gradient flow \eqref{flotgradientdiscretexplicite}--\eqref{flotgradientdiscretexplicite-boundary} with $N = 49$ particles. The minimal number of particles to form an aggregate is $k=31$.  The initial data is a random perturbation of a symmetric configuration. (Left) Blow-up occurs by merging the minimal number of particles. The blow-up time is approximately $T = 0.24$. (Right) After a parabolic rescaling, particles in the outer set are sent to infinity, and particles in the inner set converges towards a critical profile.}
\label{explosionkpamisnreg}
\end{figure}
In order to state our stability result we define a decreasing family of critical sensitivity parameters for $k = 1\dots N$,
\begin{equation}
\chi^k_N=\frac{N+1}{k}\,.
\end{equation}
It is not difficult to prove that for $\chi<\chi^k_N$, $k$ isolated particles cannot form a blow-up aggregate. Therefore it is natural to address the following problem.
\begin{Problem}[Discrete mass quantization problem.]\label{DmassQuantProblem}
Assume $\chi^k_N< \chi< \chi^{k-1}_N $. Does the first blow-up set contain exactly $k$ particles?
\end{Problem}
We shall see that in the case of three particles, answer to Problem \ref{DmassQuantProblem} is false in general for symmetry reasons. However we give below a positive answer to Problem \ref{DmassQuantProblem} for a reasonably large set of initial data.
\begin{Thm}[Stability]\label{Thm:stabilityintro}
Let $\chi^k_N< \chi< \chi^{k-1}_N $. There exists a set $D_{N,\chi}$ such that for any initial data $X(0) \in D_{N,\chi} $, the solution $X$ of \eqref{flotgradientdiscretexplicite}--\eqref{flotgradientdiscretexplicite-boundary} aggregates exactly $k$ particles when blow-up occurs. The set $D_{N,\chi}$ plays the role of a basin of stability.
\end{Thm}
The stability set $D_{N,\chi}$ is defined in \eqref{def:stabilityset}. It is parametrized by some arbitrary $\epsilon>0$, measuring the isolation of some subset of particles initially (the {\em inner} set). Furthermore this parametrization enables to contract the isolated subset of particles such that we control the blow-up time.

We are not able to handle the transition cases, where $\chi = \chi_N^k$ for some $k = 1\dots N$, as discussed in the case of three particles.

A formal stability result was obtained in the two-dimensional case by Vel\' azquez \cite{Vel02}. The author showed that a small perturbation of the initial data  leads to the formation of a singularity which is close in time and location. A precise statement is contained in \cite{Raphael-Schweyer}.

Within the framework of Theorem \ref{Thm:stabilityintro} we are able to prove quantitative results about the blow-up dynamics.
We define the  rescaled free energy.
\begin{Def}[The local rescaled energy functional]\label{energiefonctionnellepartiellei}
Let $k = 1\dots N$, and $\In\subset[1,N]$ a blow-up set such that $|\In| = k$.
We define $\E^{\rm resc}_{k}$ by:
\begin{equation*}
\E^{\rm resc}_{k}\left( Y \right)=-\sum_{i \in \In \setminus\{\max \In\}  } \log \left(Y_{i+1}-Y_i\right) +\chi h_N \sum_{ \begin{subarray}{c}
(i, j)\in \In\times \In \\ i\neq j\end{subarray}}  \log \vert Y_i-Y_j \vert -\frac{\alpha}{2}  \sum_{i \in \In} \left| Y_i \right|^2\, ,
\end{equation*}
where $\alpha = \left(k-1\right)\left(\frac{\chi} {\chi^k_N} - 1\right)>0$.
\end{Def}
\begin{Thm}[Rigidity of the blow-up]\label{Thm:rigidityintro}
Let $X\in \R^n$ such that any blow-up set is made of $k$ particles. Let $\In$ be one of them, then there exists $e_\infty\in \R$
such that for any sequences $t_n \to +\infty$ there exists a subsequence $t_n'  \to +\infty$ and a critical point $Y_\infty$ of $\E^{\rm resc}_{k}$, having energy level $\E^{\rm resc}_{k}(Y_\infty) = e_\infty$,
\[
\left( \frac{X_i\left(t_n'\right)-\bar X}{\sqrt{2\alpha\left(T- t_n'\right)}} \right)_{i\in \In} \underset{n \to +\infty}{\longrightarrow} Y_{\infty} \in C_{\infty},
\]
where $T$ is the blow-up time, $\bar X$ the blow-up point and $\alpha=-\left(k-1\right)\left(1-\frac{\chi} {\chi^k_N}\right)$.
\end{Thm}
It means that the blow-up profile involves only the $k$ particles contributing to the blow-up. In addition they all aggregate with the same parabolic rate, with an asymptotic profile up to extraction.

Theorem \ref{Thm:rigidityintro} is very much inspired by the analysis of blow-up for the nonlinear heat equation; see e.g. \cite{GK85,zaag,MS02} for classical references on this subject.

\bigskip

The paper is structured as follows: in section $2$ we explain how our problem is related to the classical Keller-Segel equation in dimension $2$. We recall some classical results regarding the blow-up phenomena. We define analogue quantities in the discrete setting. In section \ref{sec:introBU} we introduce and discuss the blow-up phenomena for the system \eqref{flotgradientdiscretexplicite}--\eqref{flotgradientdiscretexplicite-boundary}. We introduce some useful tools in Sections \ref{sec:secondmomentesti} and \ref{sec:induction}. Section \ref{sec:stability} is devoted to the stability issue and contains the proof of Theorem \ref{Thm:stabilityintro}. In section \ref{sec:rigidity} we prove Theorem \ref{Thm:rigidityintro}. Finally in section \ref{sec:conc} we discuss the perspectives of our work.

\section{Blow-up phenomena} \label{sec:introBU}
\subsection{Definitions}
For the system defined in \eqref{flotgradientdiscretexplicite}--\eqref{flotgradientdiscretexplicite-boundary} we give two definitions of the blow-up.
\begin{Def}[Blow-up of particles]\label{blowup}
Let $T\in\R_+^*$ and $X$ a solution of \eqref{flotgradientdiscretexplicite}--\eqref{flotgradientdiscretexplicite-boundary} defined on $[0,T)$. Let $\In \subset [1,N]$ a connected set of indices and $\Out =[1,N] \setminus  \In$. We say that $\In$ weakly blows up at time T if 
\begin{equation}\label{weakblowup}
\forall (i,i+1) \in \In \times \In \quad \liminf_{t\to T^- } \left( X_{i+1}-X_i \right) =0.
\end{equation}
We say that $\In$ strongly blows-up if there exist $c>0$ with 
\begin{equation}\label{strongblowup}
\begin{array}{lll}
&\forall (i,i+1) \in \In \times \In \quad  &\lim_{t\to T^- } \left(X_{i+1}-X_i\right)=0,\\
&\forall (i,j) \in \In \times \Out \quad  &\left|X_{i}-X_j\right|\geq \frac1c.
\end{array}
\end{equation}
In any case, when the set  $\In$ is maximal for the inclusion, 
we call it a blow-up set.
\end{Def}
Notice that a strong blow-up set is a weak blow-up set. The difference between both definitions is the possibility of oscillations for a weak blow-up set. It is not trivial to rule out this behaviour. In Proposition \ref{prop:weakisstrong} we show that a weak blow-up set made of $k$ particles is a strong blow-up set. 

It is natural to study the actions of dilations on the energy $\E$.
By the logarithmic homogeneity we have for $\lambda >0$  
\begin{equation}\label{Energiediscretm}
 \E\left(\lambda X\right)= \E\left(X\right) -\log\left(\lambda\right)\left[\left(N-1\right)-\chi h_N N\left(N-1\right)\right]. 
\end{equation}
We define accordingly the critical parameter $\chi_N$:
\begin{Def}[Critical parameter]
\[\chi_N= \frac{1}{h_NN}=1+\frac{1}{N}.\]
\end{Def}
The heuristics of \eqref{Energiediscretm} is the following: if $\chi < \chi_N$ then the energy is not bounded from below when $\lambda$ goes to infinity, which means a dilatation of the set of particles ($X_i$). 
It is the subcritical regime.
\smallskip
If $\chi > \chi_N$ the energy is not bounded from below when $\lambda$ goes to $0$ which corresponds to a contraction of the set of particles ($X_i$). 
Moreover the computation of the second moment $\Pi^2\left(X\right)=\sum^{N}_{i=1} X^2_{i}$ gives
\begin{equation}\label{secondmomentlevraidecroit}
\frac{1}{2}\frac{d}{dt} \Pi^2= \left<\dot X(t),X(t) \right> =\left<-\nabla \E(X(t)),X(t)\right> =\left(N-1\right)\left[1 - \chi h_NN \right]<0.
\end{equation}
The last equality is obtained by differentiation of \eqref{Energiediscretm} with respect to $\lambda$, at $\lambda=1$. 
Since $\Pi^2$ is positive this computation fails in finite time. It means that there exists $(T,i)\in \R^+\times [1,N]$ such that $\liminf_{t\to T^-} \big( X_{i+1}\left(T\right)- X_{i}\left(T\right)\big)=0$. The set $[i,i+1]$ weakly blows-up in finite time and there exists a maximal weak blow-up set.

Similarly we define the critical parameter ok $k$ adjacent particles bearing mass $h_n$: $$\chi_N^k=\frac{1}{h_Nk}.$$
We aim to give natural and robust conditions under which a maximal blow-up set $\In$ carries exactly the critical number of particles, that is $k$ such that 
$\chi^k_N < \chi < \chi^{k-1}_N $. 
 
The computation \eqref{secondmomentlevraidecroit} also proves the following claim. 
\begin{Claim}\label{claim:particulesbornees}
Let $\chi > \chi_N$ and $X$ be a solution of \eqref{flotgradientdiscretexplicite}--\eqref{flotgradientdiscretexplicite-boundary}. Then there exists a weak blow-up set and for any $i\in [1,N]$, $|X_i|\leq \sqrt{\Pi^2(0)}$.
\end{Claim}

\subsection{A first look on the blow-up structure }
We take $\chi^k_N \leq \chi < \chi^{k-1}_N$. 
\begin{Prop}[Blow-up properties]\label{blowupaumoinsk}
A weak (strong $\Rightarrow$ weak) blow-up set $\In$ contains at least $k$ particles.
\end{Prop}
This is a discrete analogous to the usual statement: the mass contained in a blow-up point is at least critical \cite{LuSuVel12,SenSuz01,Suzukibook}
\begin{proof}
We consider a weak blow-up set made of $p\leq k-1$ particles: $X_l,...,X_{l+p-1}$, we note $\In=\left[l,l+p-1\right]$ and $\Out=\left[1,N\right]\setminus \left[l,l+p-1\right]$. 
Since $\In$ is a weak blow-up set, by maximality, we have $$\min \left(  \liminf_{t\to T^- } \left( X_l-X_{l-1}\right) ,\liminf_{t\to T^- } \left(X_{l+p}-X_{l+p-1}\right)\right)>0.$$ Therefore there exists $c>0$ such that for any $j \in \Out$, $i\in \In$ and $s\in [0,T)$,
\begin{equation}\label{courteportee}
 |X_j\left(s\right)-X_i\left(s\right)|\geq \frac1c.
\end{equation} 
Let us consider the local energy 
\[
\E_p \left( X \right)=-\sum_{i \in \In \setminus \lbrace l+p-1 \rbrace  } \log \left(X_{i+1}-X_i\right) +\chi h_N \sum_{ (i, j)\in \In \times \In \setminus \lbrace i \rbrace}  \log \vert X_i-X_j \vert .
\]
Notice that $\chi_N^p$ is the critical parameter for $\E_p$.
Thanks to \eqref{courteportee} and the Young inequality we find $A>0$ such that 
\begin{align*}
\frac{d}{dt} \E_p &= -\left< \nabla \E_p,(\nabla_i \E )_{i\in \In} \right>_{\In} =-\| \nabla \E_p\|^2_{\ell^2(\In)} -\chi h_N\left< \nabla \E_p,\sum_{k\in \Out}\frac{1}{X_k-X_i} \right>_{\In} \\
&\leq -\| \nabla \E_p\|^2_{\ell^2(\In)} +\chi h_N\| \nabla \E_p\|_{\ell^2(\In)} \|\sum_{k\in \Out}\frac{1}{X_k-X_i} \|_{\ell^2(\In)}  
 \leq -\frac{1}{2} \|\nabla \E_p\|^2+A^2,
\end{align*}
and therefore for any $t>0$:
\begin{equation}\label{ekdecroissantalmost}
\E_p \left( X(t) \right)\leq \E_p \left( 0 \right)+ tA^2.
\end{equation}
Adapting the proof of the discrete logarithmic Hardy-Littlewood-Sobolev inequality given in \cite[Prop. 4.2]{BCC08}, we easily show the "non-constant-mass discrete logarithmic Hardy-Littlewood-Sobolev inequality: for $p \in \N$:
\begin{equation}\label{HLSdiscretp}
-\sum_{i \in \In \setminus \lbrace l+p-1 \rbrace  } \log \left(X_{i+1}-X_i\right) + \frac{1}{p} \sum_{ (i, j)\in \In \times \In \setminus \lbrace i \rbrace}  \log \vert X_i-X_j \vert \geq 0.
\end{equation}
We define $\theta$ such that $\frac{1}{p}=\frac{\chi h_N}{\theta}$, observe that $\theta=\frac{\chi}{\chi_N^p}<1$ since $p\leq k-1$. 
Combining \eqref{HLSdiscretp} and \eqref{ekdecroissantalmost} we obtain for any $t\in[0,T)$:
\[
-\sum_{i \in \In \setminus \lbrace l+p-1 \rbrace  } \log \left(X_{i+1}-X_i\right) \leq \frac{\E_p(0)+TA^2}{1-\theta}.
\]
The second moment, $\sum^{N}_{i=1} X^2_{i}$, decreases. Taking $A$ larger if needed we can suppose $|X_i|\leq A$. We deduce that for any $i\in [l,l+p-2]$ and $t\in[0,T)$:
\[
 \left(X_{i+1}-X_i\right) \geq \min{\left(1,e^{-\frac{\E_p(0)-TA^2}{1-\theta}-(p-2)\log(2A)}\right)}.
\]
It is a contradiction with $\In$ being a weak blow-up set and proves Proposition \ref{blowupaumoinsk}.
\end{proof}     
\begin{Rk}
The discrete discrete logarithmic Hardy-Littlewood-Sobolev \eqref{HLSdiscretp} rewrites
\begin{equation*}
-(1-h_N)\sum_{i \in \In \setminus \lbrace l+p-1 \rbrace  }h_N \log \left(X_{i+1}-X_i\right) +\frac{1}{p/N} \sum_{ (i, j)\in \In \times \In \setminus \lbrace i \rbrace}  h_N^2 \log  \vert X_i-X_j \vert \geq (p-1) C_{h_N}.
\end{equation*}
By analogy with the classical logarithmic Hardy-Littlewood-Sobolev inequality, the parameter $1-h_N$ as to be considered as a dimension parameter, whereas the coefficient $p/N$ corresponds to the total mass. 
\end{Rk}

\section{Second moment and exterior potential estimates}\label{sec:secondmomentesti}
We fix $\chi^k_N< \chi < \chi^{k-1}_N$. In order to catch the structure of the discrete Keller-Segel equation we define three important quantities.
\begin{Def}\label{stanfarddevH}
Let $\In$ be a connected set of indices (the {\em inner set}), say 
$|\In| = p+1$, and $\Out = \left[1,N\right]\setminus \In$ (the {\em outer set}).  
We introduce the variance of the family $\left(X_{\In}\right) = \{X_l,...X_{l+p}\}$ is defined as follows 
\begin{equation}
\Pi_{\In}^2 = \sum_{i\in \In}\left( X_i-\bar X_{\In}\right)^2 \, ,\quad {\rm where} \quad \bar X_{\In} = \frac1{|\In| }  \sum_{i\in \In}  X_i\, .
\end{equation}
We also introduce the following variant: for a given $\bar X \in \R$, the squared distance to $\bar X$ is defined by
$
\bar \Pi_{\In}^2 = \sum_{i\in \In}\left( X_i-\bar X\right)^2\, .
$
In the sequel, the point $\bar X$ will denote the blow-up location of the inner set $\In$. The existence of $\bar X$ will be deduced from \eqref{secondmomentdecroissant} in Proposition \eqref{prop:weakisstronweak}.

The exterior interaction potential is: 
\[
H_{\In\Out,m} = \sum_{j\in \Out}\sum_{i\in \In}\frac{1}{\left( X_j-X_i\right)^m}\, .
\]
We will essentially use $H_{\In\Out,2}$ and $H_{\In\Out,4}$.
\end{Def}
\noindent We are able to close a system of inequalities controlling the growth of these quantities.
Next Lemma compare the dynamics of the whole system of particles with the isolated set  $\In$. The idea is to consider the 
interaction with the outer set as a perturbation of the stand alone dynamics.  
\begin{Lem}\label{evolutionsecondmomentetH}
The following estimates for the evolution of $\Pi^2_{\In} $, $\bar \Pi^2_{\In} $ and $H_{\In\Out,2}$ hold true,
\begin{align}\label{secondmomentdecroissant}
 \left|\frac12 \frac{d}{dt} \Pi^2_{\In} - p\left(1- \frac{\chi}{\chi_N^{p+1}}\right) \right| & \leq \,   \left(2+\frac{2\chi}{\sqrt{N}}\right)\sqrt{\Pi^2_{\In}
 H_{\In\Out,2}}\, , \\
\label{secondmomentdecroissantbar}
 \left|\frac12 \frac{d}{dt} \bar  \Pi^2_{\In} - p\left(1- \frac{\chi}{\chi_N^{p+1}}\right)\right|   &\leq \,  \left(2+\frac{2\chi}{\sqrt{N}}\right)\sqrt{\bar \Pi^2_{\In}
 H_{\In\Out,2}}\, , \\
\label{Hcroissantdeux}
\frac{d}{dt} H_{\In\Out,2} & \leq \,   C(\chi,N)   H_{\In\Out,2}^2 \, .
\end{align}
\end{Lem}


\subsection{Proof of Lemma \ref{evolutionsecondmomentetH}}
\begin{proof}[Proof of Lemma \ref{evolutionsecondmomentetH}]
We start with the evolution of $\Pi^2_{\In}$, recalling that $X$ satisfies the differential equation 
\eqref{flotgradientdiscretexplicite}--\eqref{flotgradientdiscretexplicite-boundary}.
\begin{align}
\frac12\frac{d}{dt} \Pi^2_{\In} 
\nonumber =  \sum_{i\in \In} \left[  -\frac{ X_i- \bar X_{\In}}{X_{i+1}-X_i}+\frac{ X_i- \bar X_{\In}}{X_{i}-X_{i-1}}  + 2\chi h_N \sum_{j\neq i} \frac{ X_i- \bar X_{\In}}{X_{j}-X_i}\right] 
 - \sum_{i\in \In}\left( X_i- \bar X_{\In}\right)\frac{d}{dt} \bar X_{\In}  \nonumber \\
\nonumber= \sum_{i=\In \setminus{\{l+p\}}} \left[-\frac{ X_i-X_{i+1}}{X_{i+1}-X_i}\right] -  \frac{ X_{l+p}- \bar X_{\In}}{X_{l+p+1}-X_{l+p}} + \frac{ X_l- \bar X_{\In}}{X_{l}-X_{l-1}}   + 2\chi h_N \sum_{i\in \In}  \left[\sum_{j\neq i} \frac{ X_i- \bar X_{\In}}{X_{j}-X_i}\right]  \nonumber\\
= p-  \frac{ X_{l+p}- \bar X_{\In}}{X_{l+p+1}-X_{l+p}} + \frac{ X_l- \bar X_{\In}}{X_{l}-X_{l-1}} + 2\chi h_N \sum_{i\in \In} \left[\sum_{j\neq i} \frac{ X_i- \bar X_{\In}}{X_{j}-X_i}\right]\, .  \label{test}
\end{align}
We used $\sum_{i \in \In } X_i=|\In| \bar X_{\In} $. We first look at the contraction term:
\begin{align*}
T &=\sum_{i \in \In } \left[\sum_{j\neq i} \frac{ X_i- \bar X_{\In}}{X_{j}-X_i}\right]=
\underbrace{\sum_{i \in \In } \left[\sum_{ j\in \In\setminus \lbrace i\rbrace} \frac{ X_i- \bar X_{\In}}{X_{j}-X_i}\right]}_{T_1}+
\underbrace{\sum_{i \in \In } \left[\sum_{j\in \Out} \frac{ X_i- \bar X_{\In}}{X_{j}-X_i}\right]}_{T_2}\\
\end{align*}
The Cauchy-Schwarz inequality on the last term implies
\begin{equation*}
T_2\leq \sqrt{|\Out|\Pi^2_{\In}H_{\In\Out,2}}.
\end{equation*}
A similar estimates holds for the two boundary terms in \eqref{test}.
Using the symmetry we simplify $T_1$:
\begin{align*}
T_1&= \sum_{i\in \In }\sum_{ j\in \In \setminus \lbrace i\rbrace} \left[ \frac{ X_i- \bar X_{\In}}{X_{j}-X_i}\right]
= \frac12\sum_{i\in \In }\sum_{ j\in \In\setminus \lbrace i\rbrace} \left[ \frac{ X_i- \bar X_{\In}}{X_{j}-X_i}\right]+\frac12\sum_{j\in \In }\sum_{ i\in \In\setminus \lbrace j\rbrace} \left[ \frac{ X_j- \bar X_{\In}}{X_{i}-X_j}\right]\\
&=\frac12 \sum_{(i,j)\in \In \times \In \setminus \lbrace i=j\rbrace}\frac{ X_j-  X_{i}}{X_{i}-X_j} 
=- \frac{p\left(p+1\right)}{2}.
\end{align*}
All in one we obtain
\begin{equation*}
2\chi h_NT\leq - 2\chi h_N\frac{p\left(p+1\right)}{2}+2\chi h_N\sqrt{N\Pi^2_{\In}H_{\In\Out,2}}.
\end{equation*}
Coming back to \eqref{test} we get
$$
\frac12\frac{d}{dt}  \Pi^2_{\In} \leq p\left(1- \frac{\chi}{\chi_N^{p+1}}\right) +\left(2+\frac{2\chi}{\sqrt{N}}\right)\sqrt{\Pi^2_{\In}H_{\In\Out,2}}.
$$
Similarly 
$$
\frac12\frac{d}{dt}  \Pi^2_{\In} \geq p\left(1- \frac{\chi}{\chi_N^{p+1}}\right) -\left(2+\frac{2\chi}{\sqrt{N}}\right)\sqrt{\Pi^2_{\In}H_{\In\Out,2}},
$$
The demonstration for $\bar \Pi^2_{\In} $ is exactly the same because $ \sum_{i\in \In}\left( X_i-\bar X\right)\frac{d}{dt} \bar X=0$ since $\bar X$ is constant.\\
We now look for the evolution of the time exterior potential $H_{\In\Out,2}$.
\begin{align*}
\frac{d}{dt} H_{\In\Out,2}&= -2 \underset{j \in \Out }{\sum}\underset{i \in \In }{\sum}\frac{1}{\left(X_j-X_i\right)^3}\Bigg[ \\
& \frac{1}{\left(X_{i+1}-X_i\right)}-\frac{1}{\left(X_i-X_{i-1}\right)}-\frac{1}{\left(X_{j+1}-X_j\right)}+\frac{1}{\left(X_{j}-X_{j-1}\right)}\\
&\left.+2\chi h_N \left( -\sum_{ k\neq i } \frac{1}{\left(X_{k}-X_{i}\right)}+\sum_{ k\neq j } \frac{1}{\left(X_{k}-X_{j}\right)} \right)  \right].
\end{align*}
We split the right hand side into four terms: 
\begin{equation}\label{estHdeux}
\frac{d}{dt} H_{\In\Out,2}=A+B+C+D,
\end{equation}
where
\begin{align*}
A&=-2\sum_{j \in \Out }\sum_{i \in \In }\frac{1}{\left(X_j-X_i\right)^3}\left[ \frac{1}{\left(X_{i+1}-X_i\right)}-\frac{1}{\left(X_{i}-X_{i-1}\right)}\right],\\
B&=-2\sum_{j \in \Out }\sum_{i \in \In }\frac{1}{\left(X_j-X_i\right)^3}\left[-\frac{1}{\left(X_{j+1}-X_j\right)}+\frac{1}{\left(X_{j}-X_{j-1}\right)}\right],\\
C&= 4\chi h_N\sum_{j \in \Out }\sum_{i \in \In }\frac{1}{\left(X_j-X_i\right)^3}\left[ \sum_{ k\neq i } \frac{1}{\left(X_{k}-X_{i}\right)}\right],\\
D&=-4\chi h_N\sum_{j \in \Out }\sum_{i \in \In }\frac{1}{\left(X_j-X_i\right)^3}\left[\sum_{ k\neq j }\frac{1}{\left(X_{k}-X_{j}\right)}\right].
\end{align*}
The strategy is to bound each term from above with $H_{\In\Out,4}$.\\
A discrete integration by parts on $B$ gives  
\begin{multline}
\label{B1} B=-2 \sum_{i \in \In }\sum_{j \in \Out  } \frac{1}{\left(X_{j}-X_{j-1}\right)}\left[\frac{1}{\left(X_j-X_i\right)^3}-\frac{1}{\left(X_{j-1}-X_i\right)^3}\right] 
\\ +2\sum_{i \in \In }\frac{1}{\left(X_{l}-X_{l-1}\right)}\frac{1}{\left(X_{l-1}-X_i\right)^3}-2\sum_{i \in \In }\frac{1}{\left(X_{l+p+1}-X_{l+p}\right)}\frac{1}{\left(X_{l+p+1}-X_i\right)^3}.
\end{multline}
Since $l-1<i<l+p+1 $, we have 
\[\left(X_{l}-X_{l-1}\right)\left(X_{l-1}-X_i\right) \leq 0 \mbox{ and } \left(X_{l+p+1}-X_{l+p}\right)\left(X_{l+p+1}-X_i\right)\geq 0.\]
Therefore the contributions of the boundary, \textit{i.e.} the two last terms in \eqref{B1}, are nonpositive and can be dismissed for the upper bound of $\frac d {dt} H_{IO,2}$.  
There remains to treat the first term  of the right hand side of \eqref{B1}.
In the following computation the summation over $i$ and $j$ is taken for $i \in \In$ and $j \in \Out$.
\begin{align*}
& -2\sum_{i,j} \frac{1}{\left(X_{j}-X_{j-1}\right)}\left[\frac{1}{\left(X_j-X_i\right)^3}-\frac{1}{\left(X_{j-1}-X_i\right)^3}\right]\\
 &={2\sum_{i,j} \left[\frac{\left(X_{j-1}-X_i\right)^2+\left(X_j-X_i\right)^2+\left(X_{j-1}-X_i\right)\left(X_j-X_i\right)}{\left(X_j-X_i\right)^3\left(X_{j-1}-X_i\right)^3}\right]}\\
 &= {2\sum_{i,j}
\left[\frac{1}{\left(X_j\hspace{-.5mm}-X_i\right)^3\left(X_{j-1}\hspace{-1mm}-X_i\right)}+ \frac{1}{\left(X_j-X_i\right)\left(X_{j-1}\hspace{-.5mm}-\hspace{-.5mm}X_i\right)^3}+\frac{1}{\left(X_j\hspace{-.5mm}-X_i\hspace{-.5mm}\right)^2\left(X_{j-1}\hspace{-.5mm}-\hspace{-.5mm}X_i\right)^2}\right]}.
\end{align*}
This contribution is always positive. 
The H\"older inequality applied on each of the three terms, with coefficient $(4/3,4)$, $ (4,4/3)$ and $(2,2)$ gives
\[
2\sum_{i \in \In }\sum_{j \in \Out  } \frac{\left(X_{j-1}-X_i\right)^2+\left(X_j-X_i\right)^2+\left(X_{j-1}-X_i\right)\left(X_j-X_i\right)}{\left(X_j-X_i\right)^3\left(X_{j-1}-X_i\right)^3}
\leq 6 H_{\In\Out,4}.
\]
Coming back to $B$ we get 
\begin{equation}\label{ineB}
B\leq 6 H_{\In\Out,4}.
\end{equation}

A discrete integration by parts on $A$ gives a result similar to $B$ except for the sign of the boundary terms.
\begin{multline}
 A= -2 \sum_{i,i-1 \in \In }\sum_{j \in \Out  } \frac{1}{\left(X_{i}-X_{i-1}\right)}\left[\frac{1}{\left(X_j-X_{i-1}\right)^3}-\frac{1}{\left(X_{j}-X_i\right)^3}\right]\\
 +2\sum_{j \in \Out }\frac{1}{\left(X_{l}-X_{l-1}\right)}\frac{1}{\left(X_{j}-X_l\right)^3}-2\sum_{j \in \Out }\frac{1}{\left(X_{l+p+1}-X_{l+p}\right)}\frac{1}{\left(X_{j}-X_{l+p}\right)^3}.\label{A2}
\end{multline}
The boundary terms, \textit{i.e.} the last two terms of the right hand side of  \eqref{A2}, have no sign. 
Since $j\in \Out$ and $l+p \in \In $, the H\"older inequality applied to the last term of \eqref{A2} with coefficient $q=4$ and $q'=4/3$ implies
\[
2\left|\sum_{j \in \Out }\frac{1}{\left(X_{l+p+1}-X_{l+p}\right)}\frac{1}{\left(X_j-X_{l+p+1}\right)^3}\right|\leq 2N^{1/4} H_{\In\Out,4}.
\]
Similarly,  the second term of the r.h.s. of \eqref{A2} satisfies
\[
\left|2\sum_{j \in \Out }\frac{1}{\left(X_{l}-X_{l-1}\right)}\frac{1}{\left(X_j-X_{l}\right)^3}\right|\leq  2N^{1/4}H_{\In\Out,4}.
\] 
There remains to deal with the first term of  the right hand side of \eqref{A2}, the core of the integration by parts. 
We follow the proof done for $B$ to avoid the singularity and get

\begin{equation}\label{ineA}
A\leq \left(6+4N^{1/4}\right) H_{\In\Out,4}.
\end{equation}

Concerning $D$ we have 
\begin{align*}
D
&=\underbrace{ \sum_{j \in \Out }\sum_{i \in \In } \sum_{k\in \In }\frac{-4\chi h_N }{\left(X_{k}-X_{j}\right)\left(X_j-X_i\right)^3}}_{D_1}+\underbrace{\sum_{j \in \Out } \sum_{i \in \In } \sum_{ k\in \Out \setminus \lbrace j \rbrace }\frac{-4\chi h_N }{\left(X_{k}-X_{j}\right)\left(X_j-X_i\right)^3}}_{D_2}\\
\end{align*}
Since $j\in \Out$ and $i,k\in \In$, the contribution of $D_1$ is positive. The H\"older inequality with $p=4$, $q=4/3$ gives:
\begin{align*}
D_1
&\leq4\chi h_N \left[\sum_{j \in \Out }\sum_{i \in \In }\sum_{k\in \In }\frac{1}{\left(X_{k}-X_{j}\right)^4}\right]^{1/4} \left[\sum_{j \in \Out }  \sum_{i \in \In }\sum_{k\in \In } \frac{1}{\left(X_j-X_i\right)^4}\right]^{3/4} \\
&\leq 4\chi h_N N^{1/4}\left(H_{\In\Out,4}\right)^{1/4}N^{3/4}\left(H_{\In\Out,4}\right)^{3/4}
\leq 4\chi H_{\In\Out,4}.
\end{align*}
For $D_2$ 
we use the symmetric roles of $j$ and $k$.
\begin{align*}
D_2
&= -2\chi h_N\sum_{j \in \Out }\sum_{i \in \In }\sum_{ k\in \Out \setminus \lbrace j \rbrace }
\left[\frac{1}{\left(X_{k}-X_{j}\right)\left(X_j-X_i\right)^3}-\frac{1}{\left(X_{k}-X_{j}\right)\left(X_k-X_i\right)^3}\right]\\
&= -2\chi h_N\sum_{j \in \Out }\sum_{i \in \In }\sum_{ k\in \Out \setminus \lbrace j \rbrace }
\left[\frac{1}{\left(X_j-X_i\right)^3\left(X_k-X_i\right)}+\frac{1}{\left(X_j-X_i\right)\left(X_k-X_i\right)^3}
\right. \\ & \left.  \qquad \qquad   \qquad \qquad  \qquad \qquad  \qquad \qquad  \qquad \qquad  \qquad \qquad +\frac{1}{\left(X_j-X_i\right)^2\left(X_k-X_i\right)^2}\right].
\end{align*} 
We see that this contribution is negative when $j,k\geq i$ or $j,k\leq i$, positive elsewhere. We estimate it in all cases with an 
H\"older estimate on each of the three terms. The parameters are respectively $(q,q')=(4/3,4)$ then $(q,q')=(4,4/3)$ and $(q,q')=(2,2)$. It gives
\begin{equation}
D_2 \leq 6 \chi h_N N H_{\In\Out,4} \leq 6 \chi H_{\In\Out,4}.
\end{equation} 
Getting back to $D$ we find
\begin{equation}\label{ineD}
D \leq 10 \chi H_{\In\Out,4}.
\end{equation}
In a similar way we estimate $C$.
\begin{align*}
C
&=\underbrace{4\chi h_N\sum_{j \in \Out }\sum_{i \in \In }\sum_{ k\in \In \setminus \lbrace i \rbrace } \frac{1}{\left(X_j-X_i\right)^3}\frac{1}{\left(X_{k}-X_{i}\right)}}_{C_1}+
 \underbrace{ 4\chi h_N\sum_{j \in \Out }\sum_{i \in \In }\sum_{ k\in \Out } \frac{1}{\left(X_j-X_i\right)^3}\frac{1}{\left(X_{k}-X_{i}\right)}}_{C_2}.
\end{align*}
An H\"older inequality with parameters $(q,q')=(4/3,4)$ gives $C_2 \leq 4\chi H_{\In\Out,4}$. Using the symmetric role of $i$ and $k$ we find that the contribution of 
$C_1$ is positive and can be dismissed. It implies 
\begin{equation}\label{ineC}
C\leq 4 \chi H_{\In\Out,4}.
\end{equation}
Together \eqref{ineB}, \eqref{ineA}, \eqref{ineD} and \eqref{ineC} in \eqref{estHdeux} implies
\begin{align}
\frac{d}{dt} H_{\In\Out,2}  &\leq \left(12+14\chi +4N^{1/4} \right) H_{\In\Out,4}\\
&\leq C_{4,2}(N)\left(12+14\chi +4N^{1/4} \right)   H_{\In\Out,2}^2,
\end{align} 
where $C_{4,2}(N)$ is the sharpest constant such that $\|\cdot \|_4\leq C_{4,2}(N)  \| \cdot \|_2$, which we know exists since we consider a finite dimensional system.
\end{proof}

\subsection{Precision on the blow-up structure}
A first application of Lemma \ref{evolutionsecondmomentetH} is the following Proposition.
\begin{Prop}\label{prop:weakisstronweak}
Let $\In$ a weak blow-up set. We denote $l$ (resp. $r$) the smallest (resp. largest) indices of $\In$ and $T$ the blow-up time.
\begin{itemize}
\item The variance of $\In$, $\Pi^2_{\In}$, converges as $t$ goes to $T$. We denote by $\Pi$ its limit.
\item The inner mean value of $\In$: $\bar X_{\In} = \frac1{|\In| }  \sum_{i\in \In}  X_i$, converges as $t$ goes to $T$. We denote by $\bar X$ its limit.
\item $\In$ is a strong blow-up set if and only if $\lim_{t\to T^- } \Pi_{\In}^2 =0.$
\end{itemize}
\end{Prop}
\begin{proof} Notice that the maximality property of a weak blow-up set implies that 
there exists $c>0$ such that for any $t>0$, $|X_l-X_{l-1}|,|X_{r+1}-X_r|\geq \frac1c$. Therefore $H_{\In\Out,2} \leq  N^2c^2$.

$\bullet$ For the first assertion, notice that Claim \ref{claim:particulesbornees} implies that $\Pi_{\In}^2$ bounded. Together with $H_{\In\Out,2}\leq N^2c^2$ and estimation \eqref{secondmomentdecroissant} we deduce that $\left|\frac12 \frac{d}{dt} \Pi^2_{\In} \right|$ is bounded, since $T$ is finite it concludes the proof. 

$\bullet$ For the second assertion a simple computation of $\frac{d}{dt} \bar X_{\In} $ combined with $H_{\In\Out,2}\leq N^2c^2$ shows that $\left|\frac{d}{dt} \bar X_{\In}\right|$ is bounded by $(2+N)Nc$. Since $T$ is finite it gives the existence of $\bar X$.

$\bullet$ The third assertion is obtained by convexity:
\[
\Pi_{\In}^2\leq N |X_r-X_l|^2\leq N (N-1)\sum_{i\in \In \setminus \{r\} }|X_{i+1}-X_i|^2\leq 2N (N-1) \Pi_{\In}^2.
\]
\end{proof}

This theorem is the first step to control the oscillations of a weak blow-up set. We then give the equivalents of Definition \ref{stanfarddevH} and Proposition \ref{evolutionsecondmomentetH} for the rescaled system.

\begin{Def}\label{def:solutionrescaledone}
Let $X$ be a solution of \eqref{flotgradientdiscretexplicite}--\eqref{flotgradientdiscretexplicite-boundary}, $T$ the blow-up time, 
$\In$ a weak blow-up set 
and $\bar X $ the blow-up point. For any $i \in [1,N]$ we define $$Y_i\left( \tau(t) \right)= \frac{X_{i}\left(t\right)-\bar X }{\sqrt{2\alpha\left(T-t\right)}},$$ where $\tau\left(t\right)=-\frac{1}{\alpha}\log\left(\frac{R\left(t\right)}{R\left(0\right)}\right)$ and  $R\left(t\right)=\sqrt{2\alpha\left(T- t\right)}$.
The particle system \eqref{flotgradientdiscretexplicite}--\eqref{flotgradientdiscretexplicite-boundary} rewrites in rescaled variables as follows,
\begin{equation*}
\begin{array}{lll}
\dot Y\left(\tau \right) = -\nabla \E^{\rm resc} \left(Y\left(\tau \right)\right), &
 Y\left(0\right)= Y^0 &.
\end{array}
\end{equation*}
where
\begin{equation*}
\E^{\rm resc}\left(Y\right)=-\sum_{i=1}^{N-1} \log\left(Y_{i+1}-Y_i\right) +\chi h_N\sum_{1\le i\neq j\le N}  \log \vert Y_i-Y_j \vert -\frac{\alpha}{2}|Y|^2.
\end{equation*}

For any $q, p\in  \In$, $q<p$ we define the average $\bar Y_{q,p}=\frac{1}{p-q+1}\sum_{i=q}^p Y_i$, the pseudo inner set $\In_{q,p}=\left[q,p\right]$, the pseudo exterior set $\Out_{q,p}=\left[1,N\right]\setminus \left[q,p\right] $ and the corresponding variance and exterior squared potential by
$$
P^2_{q,p}=
\sum_{i\in \In_{q,p}}\left( Y_i -\bar Y_{q,p}\right)^2.
$$
$$
\H_{q,p,2}=
\sum_{j\in\Out_{q,p} } \sum_{i\in \In_{q,p} }\frac{1}{\left(Y_j-Y_i\right)^2}
$$
For $q=l$, $p=l+k-1$ we denote $P^2_{q,p} = P^2_{\In}$ and $\H_{q,p,2}=\H_{\In \Out,2}$.
\end{Def}
\begin{Lem}\label{evolutionsecondmomentetHr}
Under the assumptions of Definition \ref{def:solutionrescaledone}. Let $\alpha_{q,p}=\left(p-q\right)\left(1-\chi \frac{p-q+1}{N+1}\right)=\left(p-q\right)\left(1-\frac{\chi }{\chi^{p-q+1}_{N}}\right)$, we have 
\begin{equation}\label{secondmomentdecroissantr}
\left| \frac12 \frac{d}{d \tau} P^2_{q,p} -  \alpha_{q,p}-\alpha P^2_{q,p} \right| \leq \left(2+\frac{2\chi}{\sqrt{N}}\right)\sqrt{P^2_{q,p}\H_{q,p,2}}.
\end{equation}
\end{Lem}
\begin{Cor}\label{evolutionsecondmomentetHrcorrolaire}
We deduce two different estimates regarding the number of particles $p-q+1$.
\begin{enumerate}
\item If $p-q+1\leq k-1$ i.e. $\alpha_{q,p}>0$ and $\sqrt{P^2_{q,p}\H_{q,p,2}} \leq \frac{\alpha_{q,p}}{2\left(2+\frac{2\chi}{\sqrt{N}}\right)}$ then
$$\frac12 \frac{d}{d \tau} P^2_{q,p}\geq \frac{\alpha_{q,p}}{2} + \alpha P^2_{q,p}.$$
\item If $p-q+1\geq k$ i.e. $\alpha_{q,p}<0$ and $\sqrt{P^2_{q,p}\H_{q,p,2}} \leq -\frac{\alpha_{q,p}}{2\left(2+\frac{2\chi}{\sqrt{N}}\right)}$ then
$$\frac12 \frac{d}{d \tau} P^2_{q,p}\leq -\frac{\alpha_{q,p}}{2} + \alpha P^2_{q,p}.$$
\end{enumerate}
\end{Cor}

\begin{proof}
The proof is a direct computation similar to the proof of Lemma \ref{evolutionsecondmomentetH}. 
The only difference is that an additional term pops up for $\frac{d}{d \tau} P^2_{q,p}$ : $\alpha \sum_{i\in \In_q^p}\left( Y_i -\bar Y_{q,p}\right) Y_i$. 
To deal with it we remark that
\begin{align*}
\sum_{i\in \In_q^p}\left( Y_i -\bar Y_{q,p}\right) Y_i
=\sum_{i\in \In_q^p}\left( Y_i -\bar Y_{q,p}\right)^2 =P^2_{q,p},
\end{align*}
\end{proof}

\section{Induction}\label{sec:induction}
Let $\In$ be a weak blow-up set, and $q,p\in \In$. The main difficulty to obtain rigidity theorem is to control the possible oscillations of the rescaled (resp. non rescaled) system. To do this
we proceed by induction on $q<p$ to control, from below, the partial variance of all subsets of inner particles.
The inductive argument is the following. We face the following alternative: either the variance of all but the right-most particle is large, and we are done; or it is small, and the two right-most particles are far from each other. The last statement implies that the variance of all but the right-most particle increases. Consequently the partial variance cannot be too small. The two followings propositions are the tools to develop this argument for the rescaled (resp. non rescaled) system. 
\begin{Prop}[Induction]\label{inductionnonrescale}
Let $X$ be a solution of \eqref{flotgradientdiscretexplicite}--\eqref{flotgradientdiscretexplicite-boundary} and $\In$ a weak blow-up set made of $k$ particles. Let $q,p\in \In$ and define $\Pi^2_{q,p}=\Pi^2_{[q,p]}$. Suppose there exists  $t_0>0$ and $B>0$ such that
\begin{align*}
&
\forall t\in [t_0,T), \quad \begin{cases}
\vspace{1mm} |X_{p}-X_{p-1}|\geq  \frac{1}{B} \\
|X_{q}-X_{q-1}|\geq \frac{1}{B}
\end{cases} \; \text{(reinitialization step)},\\
\text{or}\\
&
\forall t\in [t_0,T), \quad  \begin{cases} \Pi^2_{q,p}  \geq \frac{1}{B^2} \\
|X_{q}-X_{q-1}|\geq \frac{1}{B} \end{cases} \; \text{(descent step)}.
\end{align*}
Then there exists $\bar B>0$, depending only on $B$, $\Pi^2_{q,p-1}(t_0)$, $N$, $\chi$ such that
\begin{equation}
\forall t\in [t_0,T), \quad \Pi^2_{q,p-1}>\frac{1}{\bar B^2}.
\end{equation}
\end{Prop}

\begin{Prop}[Rescaled Induction]\label{induction}
Let $X$ be a solution of \eqref{flotgradientdiscretexplicite}--\eqref{flotgradientdiscretexplicite-boundary} and $\In$ a weak blow-up set made of $k$ particles. Under the assumptions of Definition \ref{def:solutionrescaledone}, let $q,p\in \In$ and suppose there exists $\tau_0>0$ and $B>0$ such that
\begin{align*}
&
\forall \tau>\tau_0, \quad \begin{cases}
\vspace{1mm} |Y_{p}-Y_{p-1}|\geq  \frac{1}{B} \\
|Y_{q}-Y_{q-1}|\geq \frac{1}{B}
\end{cases} \; \text{(reinitialization step)},\\
\text{or}\\
&
\forall \tau>\tau_0, \quad  \begin{cases} P^2_{q,p}  \geq \frac{1}{B^2} \\
|Y_{q}-Y_{q-1}|\geq \frac{1}{B} \end{cases} \; \text{(descent step)}.
\end{align*}
Then there exists $\bar B>0$, depending only on $B$, $P^2_{q,p-1}(\tau_0)$, $N$, $\chi$ such that
\begin{equation}
\forall \tau>\tau_0, \quad P^2_{q,p-1}>\frac{1}{\bar B^2}.
\end{equation}
\end{Prop}
To illustrate the proof in both cases we refer to figures \ref{figsckpnr} and  \ref{figsckpnd}.
\begin{proof}[Proof of proposition \ref{induction}] We distinguish between the descent case and the reinitialization step.
\subsection*{1- The reinitialization step.}
\begin{figure}
\includegraphics[scale=0.7]{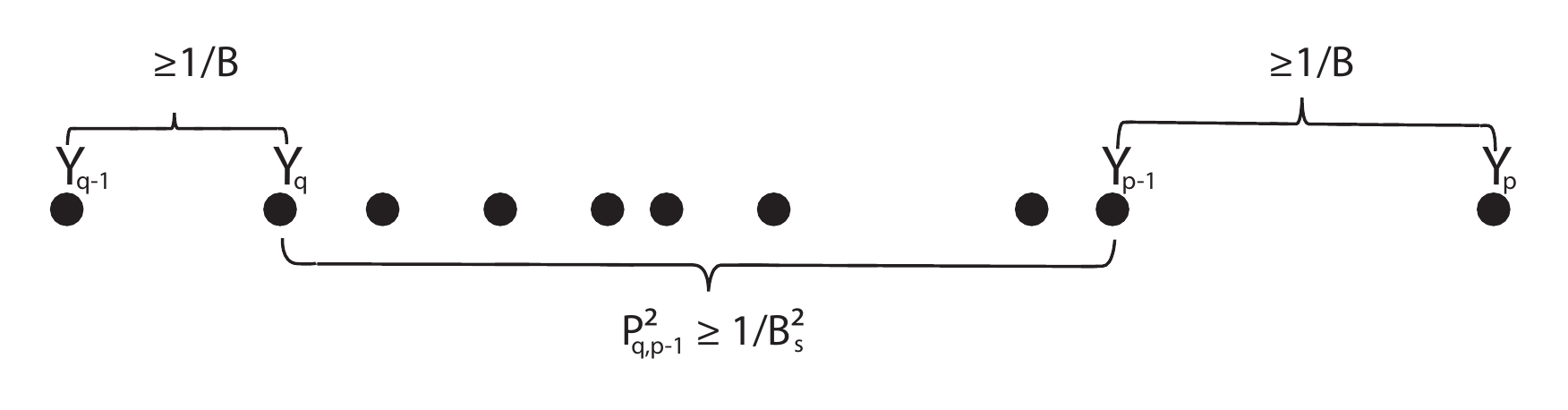}
\caption{$p-q+1<k$. In the reinitialization step, the lower bound of the extremal relative distances implies that $P^2_{q,p-1}$ cannot be too small }\label{figsckpnr}
\end{figure}
In this case we can bound from above the exterior potential $\H_{q,p-1,2}$ and we deduce that the variance $P^2_{q,p-1}$ stays away from $0$. We have,
\begin{align}
\H_{q,p-1,2}&= \sum_{j\in\Out_{q,p-1} } \sum_{i\in \In_{q,p-1} }\frac{1}{\left(Y_j-Y_i\right)^2}\\
&\leq N^2 \min \left(\frac{1}{|Y_{q-1}-Y_q|^2} ,\frac{1}{|Y_{p}-Y_{p-1}|^2}\right)
\leq N^2 B^2.
\end{align}
Furthermore, as long as
\[
P^2_{q,p-1} \leq \frac{1}{N^2B^2}\left( \frac{\alpha_{q,p-1}}{2\left(2+\frac{2\chi}{\sqrt{N}}\right)} \right)^2,
\]
we have
\[
\sqrt{P^2_{q,p-1}\H_{q,p-1,2}} \leq \frac{\alpha_{q,p-1}}{2\left(2+\frac{2\chi}{\sqrt{N}}\right)}\, .
\]
Plugging this into Corollary \ref{evolutionsecondmomentetHrcorrolaire}, together with $p-q \leq k-1$, we get that the variance $P^2_{q,p-1}$ increases:
\[
\frac12 \frac{d}{d \tau} P^2_{q,p-1} \geq \frac{\alpha_{q,p-1}}{2} + \alpha P^2_{q,p-1}>0.
\]



We easily deduce the existence of $\bar B$.
\[
P^2_{q,p-1} \geq \min \left( P^2_{q,p-1}\left( \tau_0 \right), \frac{1}{N^2B^2}\left( \frac{\alpha_{q,p-1}}{2\left(2+\frac{2\chi}{\sqrt{N}}\right)} \right)^2 \right)=\frac{1}{\bar B^2}
\]

\subsection*{2- The descent step.}
\begin{figure}
\includegraphics[scale=0.7]{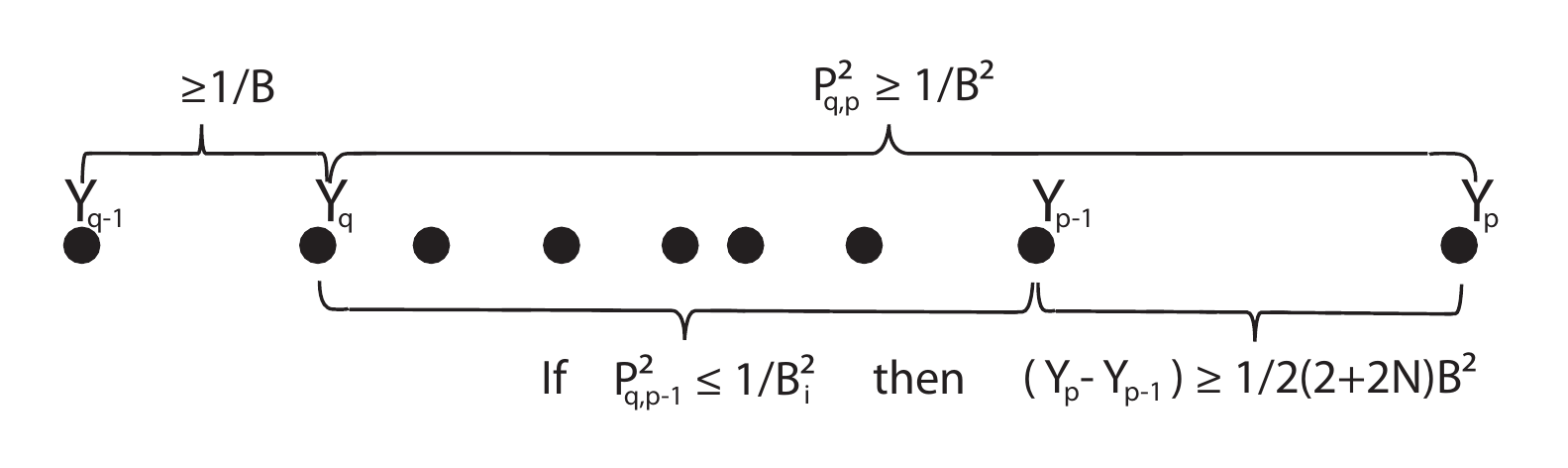}
\caption{$p-q+1<k$. In the descent step, since $P^2_{q,p}$ is large then both $P^2_{q,p-1}$ and $\left( Y_{p} - Y_{p-1}\right)^2$ cannot be too small.}\label{figsckpnd}
\end{figure}
In this case we are not able to bound directly $\H_{q,p-1,2}$ from above. Alternatively we  show that, under the condition that $P^2_{q,p-1}$ is small and  $P^2_{q,p}$ is large, we get such an estimate. First, we make a link between $P^2_{q,p-1}$, $P^2_{q,p}$ and $\left( Y_{p} - Y_{p-1}\right)^2$:
\begin{align}
\nonumber P^2_{q,p}&=\sum_{i\in \In_{q,p}}\left( Y_i -\bar Y_{q,p}\right)^2=\left( Y_{p} -\bar Y_{q,p}\right)^2+\sum_{i=q}^{p-1}\left( Y_i-\bar Y_{q,p-1}+\bar Y_{q,p-1} -\bar Y_{q,p}\right)^2\\
\nonumber &=\left( Y_{p} -\bar Y_{q,p}\right)^2+\sum_{i\in \In_{q,p-1}}\left( Y_i-\bar Y_{q,p-1}\right)^2+\sum_{i\in \In_{q,p-1}}\left( \bar Y_{q,p-1} -\bar Y_{q,p}\right)^2\\
\nonumber &+2\sum_{i\in \In_{q,p-1}}\left( Y_i-\bar Y_{q,p-1}\right)\left(\bar Y_{q,p-1} -\bar Y_{q,p}\right)\\
\label{calcul1} &=\left( Y_{p} -\bar Y_{q,p}\right)^2+P^2_{q,p-1}+2\left(p-q\right)\left( \bar Y_{q,p-1} -\bar Y_{q,p}\right)^2
\end{align}
By convexity we have
\begin{align}
\nonumber \left( \bar Y_{q,p} -\bar Y_{q,p-1}\right)^2&\leq \left( Y_{p} -\bar Y_{q,p-1}\right)^2 \leq \left( Y_{p} - Y_{p-1} + Y_{p-1}-\bar Y_{q,p-1}\right)^2\\
\nonumber &\leq 2\left( Y_{p} - Y_{p-1}\right)^2+2\left(Y_{p-1}-\bar Y_{q,p-1}\right)^2\\
\label{calcul2} &\leq 2\left( Y_{p} - Y_{p-1}\right)^2+2P^2_{q,p-1},
\end{align}
and
\begin{equation}\label{calcul3}
\left( Y_{p} -\bar Y_{q,p}\right)^2\leq \left( Y_{p} -\bar Y_{q,p-1}\right)^2 \leq 2\left( Y_{p} - Y_{p-1}\right)^2+2P^2_{q,p-1}.
\end{equation}
Plugging \eqref{calcul2} and \eqref{calcul3} in \eqref{calcul1} we obtain
\begin{equation}\label{calcul4}
\frac{1}{B^2} \leq P^2_{q,p}\leq \left(3+4N\right)P^2_{q,p-1}+\left(2+4N\right)\left( Y_{p} - Y_{p-1}\right)^2
\end{equation}
Using \eqref{calcul4} we see that $P^2_{q,p-1}$ and  $\left( Y_{p} - Y_{p-1}\right)^2$ cannot be small at the same time.
Precisely for any $B_i>0$ two cases may happen: either $P^2_{q,p-1}\geq \frac{1}{B_i^2}$ or $P^2_{q,p-1}\leq \frac{1}{B_i^2}$.
In the latter the Equation \eqref{calcul4} gives a lower bound for $\left( Y_{p} - Y_{p-1}\right)^2$.
\begin{equation*}
\left( Y_{p} - Y_{p-1}\right)^2 \geq \frac{1}{2+4N} \left( \frac{1}{B^2}-\frac{3+4N}{B_i^2} \right).
\end{equation*}
Taking $B_i$ large enough, for example $B_i^2 \geq 2 \left( 3+4N \right) B^2 $, we obtain
\begin{equation*}\label{calculborneendessous}
\left( Y_{p} - Y_{p-1}\right)^2 \geq \frac{1}{2(2+4N)} \frac{1}{B^2}.
\end{equation*}
On the other side of the pseudo inner set $\In_{q,{p-1}}$, the hypothesis is $\left( Y_{q} - Y_{q-1}\right)^2 \geq \frac{1}{B^2}$.
We deduce an upper bound for $\H_{q,p-1,2}$ and therefore an upper bound for $\sqrt{P^2_{q,p-1}\H_{q,p-1,2}}$. Similarly as in the reinitialization step:
\begin{align*}
\H_{q,p-1,2}
\leq N^2\left( B^2+ 2(2+4N)B^2 \right)\leq  N^2\left(5+8N\right)B^2.
\end{align*}
Then taking $B_i$ larger if needed, such that $B_i\geq N \sqrt{5+4N}B \frac{2\left(2+\frac{2\chi}{\sqrt{N}}\right)}{\alpha_{q,p-1}} $, we get
\begin{align}
\nonumber \sqrt{P^2_{q,p-1}\H_{q,p-1,2}}&\leq \frac{1}{B_i} N \sqrt{5+4N}B \\
\label{majrationpourdecroicasinduction}&\leq \frac{\alpha_{q,p-1}}{2\left(2+\frac{2\chi}{\sqrt{N}}\right)}.
\end{align}
Thus, the hypotheses of Corollary \ref{evolutionsecondmomentetHrcorrolaire} are fulfilled, it give that $P^2_{q,p-1}$ increase. \\
{\bf In any case} either $P^2_{q,p-1}$ is large or $P^2_{q,p-1}$ increases. We  deduce a lower bound for $P^2_{q,p-1}$:
\begin{align*}
P^2_{q,p-1} &\geq \min \left( P^2_{q,p-1}(\tau_0), \frac{1}{B_i^2} \right)\\
           &\geq \min \left( P^2_{q,p-1}(\tau_0), \frac{1}{2 \left( 3+4N \right) B^2}, \frac{\alpha_{q,p-1}}{4\left(2+\frac{2\chi}{\sqrt{N}}\right)^2 N^2 \left( 5+4N \right) B^2} \right)=\frac{1}{\bar B^2}.
\end{align*}
This proves the descent step of Proposition \ref{induction}, and finishes the proof of this proposition.
\end{proof}

\begin{proof}[Proof of Proposition \ref{inductionnonrescale}]
Mutatis Mutandis the proof is exactly the same as the one done for Proposition \ref{induction}. Using Lemma \ref{secondmomentdecroissant} instead of Corollary \ref{evolutionsecondmomentetHrcorrolaire}.
\end{proof}

\section{Stability}\label{sec:stability}
In this section we fix $\chi^k_N< \chi < \chi^{k-1}_N$. Then we exhibit stable sets of $k$ strongly blowing-up particles.
Our strategy is to obtain estimates on the blow-up time by showing that even if the problem is non linear and non local we can
focus on an isolated subset of particles, called the inner set, for which the dynamics is almost local.
\subsection{Rigidity for $k$ particles}
We start with a rigidity proposition.
\begin{Prop}[Weak is Strong]\label{prop:weakisstrong}
A weak blow-up set made of $k$ particles is a strong blow-up set.
\end{Prop} 
\begin{proof}
This proof is an ersatz of the proof of Theorem \ref{estimationbornesuperieur} done in section \ref{sec:rigidity}.
Let $\In$ be a weak blow-up set of $k$ particles. According to Proposition \ref{prop:weakisstronweak}, $\lim_{t\to T^- } \Pi_{\In}^2$ exists, let $\Pi$ be the limit. If $\Pi=0$ Proposition \ref{prop:weakisstronweak} says that $\In$ is a strong blow-up set.

Assume by contradiction that $\Pi>0$. It implies that $\Pi^2_{\In}$ is bounded from below, say by $\frac{1}{A^2}$. Thanks to the induction procedure described in Proposition \ref{inductionnonrescale}, we are able to isolate the left-most relative distance using the descent case of proposition \ref{induction}.

Let $\In=[l,l+k-1]$.
Since $\In$ is a weak blow-up set, taking $A$ larger if needed, the maximality property implies $|X_l-X_{l-1} |\geq \frac{1}{A}$. With $q=l$ and $p=l+k-1$ we are exactly in the descent case of Proposition \ref{inductionnonrescale}. It gives us $\bar B_1>A>0$ such that $\Pi^2_{l,l+k-2}\geq \frac{1}{\bar B_1^2}$. Since $|X_l-X_{l-1}| \geq \frac{1}{A} \geq \frac{1}{\bar B_1}$ we apply the descent case of proposition \ref{induction} again, with $q=l$ and $p=l+k-2$, in order to gain an additional notch on the $p$ index. We repeat the same argument for $p$ down to $p=l+2$. We obtain finally a lower bound, say $\frac{1}{\bar B^2}$, for $\Pi^2_{l,l+1}$. It gives us a lower bound for $|X_{l+1}-X_{l}|$:
\begin{align*}
\left(  X_{l+1}-X_{l} \right)^2 
\geq \Pi^2_{l,l+1}
\geq \frac{1}{\bar B^2}.
\end{align*}
This is a contradiction with $\liminf_{t\to T^- } |X_{l+1}-X_{l}| =0$. It proves that $\In$ is a strong blow-up set.
\end{proof}

\subsection{Proof of Theorem \ref{Thm:stabilityintro}}

We first give a more precise version of Theorem \ref{Thm:stabilityintro}. We exhibit below basins of attractions where $k$ particles only will be aggregated in a strong blow-up set. We define for $\chi_N^{k}< \chi < \chi_N^{k-1} $:
\begin{equation}\label{def:stabilityset}
D^{\epsilon,\frac{c}{\epsilon}}_{N,\chi}  =\lbrace X \in \R^N \mbox{ such that } \exists \,  \In,   \mbox{ with } |\In|=k, \, \Pi^2_{\In} \le \epsilon, \mbox{ and } H_{\In\Out,2} < \frac{c}{\epsilon} \rbrace ,
\end{equation}
where $\In = [l,l+k-1]$ and $\Out =[1,N] \setminus \In$.\\
This set corresponds to $k$ particles being close to each other, and all the other one being far from $\In$, but with relative distances of the same order of magnitude $O(\sqrt{\epsilon})$. Furthermore let
\begin{equation}\label{cnavantledebut}
C_N \leq \min  \left( \frac{\left(k-1\right)\left(\frac{\chi} {\chi^k_N}-1\right)}{2C_{4,2}\left(12+14\chi +4N^{1/4} \right)},\\
\frac18 \frac{\left(k-1\right)^2\left(\frac{\chi} {\chi^k_N}-1\right)^2}{\left(2+\frac{2\chi}{\sqrt{N}}\right)}  \right).
\end{equation}
We also set $\alpha=-\left(k-1\right)\left(1-\frac{\chi} {\chi^k_N}\right)>0$ and $\beta =4C_{4,2}\left(12+14\chi +4N^{1/4} \right)C_N^2$.

\begin{Thm}\label{Thm:stability}
Let $\chi_N^{k+1}< \chi < \chi_N^k $. Suppose
there exists $t_0 \in\left[0,T\right)$ such that $ X\left(t_0 \right) \in D^{\epsilon,\frac{C_N}{\epsilon}}_{N,\chi} $.
Then
\begin{equation}\label{attractiondeux}
\forall s\in[0,T-t_0), \quad X\left(t_0+s\right)\in D^{\epsilon-s{\alpha},\frac{C_N}{\epsilon}+s\frac{\beta}{\epsilon^2}}_{N,\chi}.
\end{equation}
Moreover one the two following items holds:
\begin{itemize}
\item[(i)] none of the particles in $\In$ contributes to the blow-up,
\item[(ii)] $\In$ is a strong blow-up set.
\end{itemize}
In particular if \vc{} there exists $i_0\in \In $ with $\liminf_{t\to T^- } |X_{i_0+1}-X_{i_0}| =0$
then $\In$ is a strong blow-up set aggregating exactly $k$ particles.
\end{Thm}
\begin{Rk}
We can see the sequence $D^{\epsilon-s\frac{\alpha}{2},\frac{C_N}{\epsilon}+s\frac{\beta}{\epsilon^2}}_{N,\chi}$ as a Lyapunov function over sets.
The set $D^{\epsilon,\frac{C_N}{\epsilon}+s\frac{\beta}{\epsilon^2}}_{N,\chi} $ are basins of attraction.
\end{Rk}
The idea of the proof is to show that we control the interaction
bewteen the inner and the outer set over a sufficiently long period of time,
to ensure that the blow-up effectively happens.
\begin{proof}
We show that $\Pi^2_{\In}$ decreases at least linearly on $[t_0,T)$ 
whereas $H_{\In\Out,2}$ remains bounded. Our starting point is the equation \ref{secondmomentdecroissant} of Lemma \ref{evolutionsecondmomentetH}.
$$
\frac12\frac{d}{dt}   \Pi^2_{\In} \leq -  \alpha +\left(2+\frac{2\chi}{\sqrt{N}}\right)\sqrt{\Pi^2_{\In}H_{\In\Out,2}}.
$$
Thus, as long as
\begin{equation}\label{conditiondecroissance}
\sqrt{\Pi^2_{\In}H_{\In\Out,2}} \leq \frac12 \frac{\alpha}{\left(2+\frac{2\chi}{\sqrt{N}}\right)}\, ,
\end{equation}
we get
\begin{equation}\label{hdirectzero}
\frac12\frac{d}{dt} \Pi^2_{\In} \leq -\frac{\alpha}{2} \, .
\end{equation}
Integrating \ref{hdirectzero} from $t_0$ to $t_0+s$ we get
$$
0\leq \Pi^2_{\In}\left(t_0+s\right)\le  \Pi^2_{\In}\left(t_0\right) -{\alpha} s \leq \epsilon -\alpha s.
$$
Therefore under the condition \ref{conditiondecroissance} we find an upper bound for the blow-up time $T$.
\begin{equation}\label{bornesuperieurtempsexplosion}
T \leq t_0 + \frac{\epsilon}{\alpha}.
\end{equation}
Naturally, the next step is to prove that starting at time $t_0$ with $X(t_0) \in D^{\epsilon,\frac{C_N}{\epsilon}}_{N,\chi}$ the estimate \eqref{conditiondecroissance} remains true for any $s\in [t_0,T)$. We already know that under the condition \eqref{conditiondecroissance} the second moment decreases, so it suffices to prove that $H_{\In\Out,2}$ remains  bounded
as in \eqref{conditiondecroissance}
up to $T$.

Since $X\left(t_0\right) \in D^{\epsilon,\frac{C_N}{\epsilon}}_{N,\chi}$ we have $$H_{\In\Out,2}^2(t_0)\leq \frac{C_N}{\epsilon}\leq \frac{2C_N}{\epsilon}.$$
Moreover thanks to the equation  \ref{Hcroissantdeux} of Lemma \ref{evolutionsecondmomentetH} we control the growth of $H_{\In\Out,2}$.
$$
\frac{d}{dt} H_{\In\Out,2} \leq C_{4,2}(N)\left(12+14\chi +4N^{1/4} \right)   H_{\In\Out,2}^2= \gamma_N H_{\In\Out,2}^2.
$$
Therefore for any  $s \in[0,\min(T,\frac{\epsilon}{\gamma_N C_N})]$,
$$
H_{\In\Out,2}(t_0+s)\leq \frac{1}{\frac{1}{H_{\In\Out,2}(t)}-\gamma_N s } \leq \frac{1}{\frac{\epsilon}{C_N}-\gamma_N s }.
$$

Consequently for any $s \leq \frac{\epsilon}{2C_N\gamma_N}=\frac{\epsilon}{2C_N C_{4,2}(N)\left(12+14\chi +4N^{1/4} \right)}$,
\begin{equation}\label{sassezpetit}
H_{\In\Out,2}(t_0+s) \leq \frac{2C_N}{\epsilon}.
\end{equation}

According to \eqref{conditiondecroissance} and \eqref{bornesuperieurtempsexplosion}, to conclude the proof it is enough to ensure that
$$
\sqrt{2C_N} \leq \frac12 \frac{\alpha}{\left(2+\frac{2\chi}{\sqrt{N}}\right)}.
$$
and
$$
\frac{\epsilon}{\alpha} \leq  \frac{\epsilon}{2C_{4,2}\left(12+14\chi +4N^{1/4} \right)C_N}.
$$
This is precisely the definition of $C_N$ \eqref{cnavantledebut}.

At time $T$, two cases may appear: either the variance of the $\In$ particles is equal to $0$ or it is positive. In the latter case there should exist other weak blow-up sets by the very definition of $T$.
Let  $\In_0$ be one of them. Since $H_{\In\Out,2}$ remains bounded and $\In_0$ is connected, then provided that $\Pi^2_{\In}(T)>0$, the following intersection is empty: $\In_0\cap \In=\emptyset$. In this case none of the particles of $\In$ contributes to the blow-up. This is item $(i)$ of Theorem \ref{Thm:stability}.

On the other hand, if $\Pi^2_{\In}(T)=0$, then the bound on $H_{\In\Out,2}$ and Proposition \ref{prop:weakisstrong} imply that $\In$ is a strong blow-up set. This is item $(ii)$ of Theorem \ref{Thm:stability}.
%



Finally, if there exist $i_0\in \In$ with $\liminf_{t\to T^- } |X_{i_0+1}-X_{i_0}| =0$, we are in the second case of the alternative: $\In$ is a strong blow-up set. It concludes the proof of Theorem \ref{Thm:stability}.
\end{proof}
\begin{Rk}
We do not consider the specific case where $\chi=\chi_N^{k}$ for some $k$. Although blow-up occurs in finite time (except for $k = N$), the variance of $k$ particles does not necessarily decrease. 

\end{Rk}
\begin{Rk}
It is straightforward  to construct an open set of initial data for the system \eqref{flotgradientdiscretexplicite}--\eqref{flotgradientdiscretexplicite-boundary} such that the method of Theorem \ref{Thm:stability} leads to strong blow-up with exactly $k$ particles. For example we can alternate subsets of $k$ particles lying in $D^{\epsilon,\frac{C_N}{\epsilon}}$ and subsets of less than $k$ particles. The latter cannot contribute to the blow-up according to Proposition \ref{blowupaumoinsk}.

%
\end{Rk}

\section{Rigidity}\label{sec:rigidity}
In this Section we demonstrate that 
the blow-up process including $k$ particles is rigid in the following sense: particles in the inner set $\In$ blow-up all together with the same rate, whereas particles in the outer set $\Out$ stay away from the blow-up point.

\subsection{The rescaled sytem}

Recall the parabolic rescaling that is performed in order to capture the blow-up profile:
\begin{equation}\label{renormalisationN}
Y\left(\tau\left(t\right)\right)= \frac{X\left(t\right)-\bar X}{R\left(t\right)},
\end{equation}
where 
$R\left(t\right)=\sqrt{2\alpha\left(T- t\right)}$ and
$\tau\left(t\right)=-\frac{1}{\alpha}\log\left(\frac{R\left(t\right)}{R\left(0\right)}\right)$.
The particle system \eqref{flotgradientdiscretexplicite}--\eqref{flotgradientdiscretexplicite-boundary} rewrites in rescaled variables as follows,
\begin{equation}\label{flotgradientdiscretr}
\begin{array}{lll}
\dot Y\left(\tau \right) = -\nabla \E^{\rm resc} \left(Y\left(\tau \right)\right) &
 Y\left(0\right)= Y^0 &.
\end{array}
\end{equation}
where
\begin{equation}\label{energydiscretr}
\E^{\rm resc}\left(Y\right)=-\sum_{i=1}^{N-1} \log\left(Y_{i+1}-Y_i\right) +\chi h_N\sum_{1\le i\neq j\le N}  \log \vert Y_i-Y_j \vert -\frac{\alpha}{2}|Y|^2.
\end{equation}
We can write it explicitly, with the convention $\frac{1}{Y_{1}-Y_0}=\frac{1}{Y_{N+1}-Y_N}=0$ :
\begin{equation}\label{flotgradientrNd}
\dot Y_i =-\frac{1}{Y_{i+1}-Y_i} + \frac{1}{Y_{i}-Y_{i-1}}+  2\chi
h_N\sum_{j \neq i } \frac{1}{Y_j-Y_i}+  \alpha Y_i.
\end{equation}
The center of mass $c_y =\sum_1^n Y_i $ satisfies $\dot c_y=\alpha c_y$. We cannot restrict ourselves to $c_y\left(0\right)=0$, since its value is determined through the knowledge of $\bar X$. 

\subsection{Preliminary estimates}
\begin{Thm}\label{estimationbornesuperieur}
Let $X$ be a solution of \eqref{flotgradientdiscretexplicite}--\eqref{flotgradientdiscretexplicite-boundary} 
Assume there exists a strong blow-up set of $k$ particles (for data in the basin of stability \eqref{attractiondeux} for instance).
We denote by $T$ the blow-up time, $\bar X$ the blow-up point and $Y$ the rescaled solution given by \eqref{renormalisationN}. Then there exists $A>0$ 
such~that for any $\tau>0$: 
\begin{enumerate}
\item  $ \frac{1}{A^2} \leq \| Y(\tau) \|^2_{l^2(\In)} \leq A^2 $, \quad  $ \frac{1}{A^2} \leq P^2_\In(\tau) \leq A^2 $,
\item  $\forall i   \in \In \quad \left| Y_i\left( \tau \right) \right| \leq A$,
\item $\forall (i,j) \in \In\times \In \quad  \frac1A\leq {|Y_{i}\left(\tau \right)-Y_{j}\left(\tau \right)| }\leq A$,
\item $\forall j \in \Out \quad \left| Y_j\left( \tau \right) \right| \geq \frac{1}{A\sqrt{2\alpha T}}  e^{\alpha \tau} $.
\end{enumerate}
\end{Thm}


This theorem means that, when zooming around $\bar X$ with the parabolic rate $\sqrt{2\alpha(T -t)}$, the inner particles remain bounded, whereas the outer particles are sent to $\infty$, see Figure \ref{explosionkpamisnreg} for a numerical illustration. The third estimate has an important consequence: the free energy of the inner set in the rescaled frame is bounded from below. 

\begin{Rk}
Statements of Theorem \ref{estimationbornesuperieur} are stronger than the maximum principle.
\end{Rk}

\begin{proof}[Proof of Theorem \ref{estimationbornesuperieur} ]
We split the proof into several estimates, corresponding to the different items of Theorem \ref{estimationbornesuperieur}.

\subsection*{Estimate 1- The squared distance to the blow-up point is estimated from above and below.}
We begin with the first estimate. Notice that $ \| Y(\tau) \|^2_{l^2(\In)}=\frac{\bar\Pi^2_\In(t)}{2\alpha\left(T-t\right)}$ and $P^2_\In(\tau)=\frac{\Pi^2_\In(t)}{2\alpha\left(T-t\right)}$.
By \ref{secondmomentdecroissantbar} of Lemma \ref{evolutionsecondmomentetH}, with $p+1=k$ and $\alpha =- \left(k-1\right)\left(1-\frac{\chi}{\chi_N^k}\right)$. We have
$$
|\frac12\frac{d}{dt} \bar\Pi^2_\In +\alpha| \leq \left(2+\frac{2\chi}{\sqrt{N}}\right)\sqrt{\bar\Pi^2_\In H_{\In\Out,2}}.
$$
Since $H_{\In\Out,2}$ is bounded and $\bar\Pi^2_\In \to 0$ as  $t \to T $, we have
\[
\frac{d}{dt} \bar \Pi^2_\In(t) \underset{t\to T}{\longrightarrow} -2\alpha. 
\]
It gives $\bar \Pi^2_\In(t) \sim {2\alpha\left(T-t\right)}$ as $t\to T$. We deduce the existence of $A$ as claimed in (1).
The proof is exactly the same when $\bar\Pi^2_\In$ is  replaced by $\Pi^2_\In$.
\subsection*{Estimate 2- In the blow-up set the rescaled solution is bounded from above.}
It is a straightforward consequence of item $(1)$ of Theorem \ref{estimationbornesuperieur}:
\begin{equation}\label{estimation1}
\forall i \in \In, \quad  \left|  Y_i\right| \leq\| Y \|_{l^2(\In)} \leq A.
\end{equation}
\subsection*{Estimate 4- the rescaled particles in the outer set go to infinity.}
We prove estimate $4$ now as it is a prerequisite for the proof of the third estimate. The key tool is the upper bound on $H_{\In\Out,2}$.
By hypothesis $\In$ is a strong blow-up set\vc{,} thus there exists $A$ such that 
\begin{equation}\label{bordsontgrand}
\min \left( |Y_{l+k}-Y_{l+k-1}|,|Y_{l}-Y_{l-1}| \right)  \geq \frac{1}{\sqrt{2\alpha\left(T-t\right)}A}.
\end{equation}
In particular \eqref{bordsontgrand} says that both $|Y_{l+k}-Y_{l+k-1}|$ and $|Y_{l}-Y_{l-1}|$ are bounded from below. This remark will be useful during the proof of the third estimate.
The second estimate implies $\max \left( |Y_{l+k-1}|,|Y_{l}| \right)\leq A$. So taking $A$ larger if needed we find
\begin{equation*}
\forall j \in \Out \quad \left| Y_j \right| \geq \min \left( |Y_{l+k}|,|Y_{l}| \right)\geq \frac{1}{\sqrt{2\alpha\left(T-t\right)}A}= \frac{1}{A\sqrt{2\alpha T}}  e^{\alpha \tau}.
\end{equation*}
In particular all the rescaled particles in $\Out $ are sent to infinity. This fact is also true for a weak blow-up set.
\subsection*{Estimate 3a- the rescaled relative distances in the inner set are bounded from above.}

This estimate is an immediate consequence of  \eqref{estimation1}, 
\begin{equation}\label{estimationrelatifpardessus}
\forall \left(i,j\right) \in \In \times \In, \quad |Y_{i}- Y_j| \leq |Y_{i}|+|Y_{j}| \leq 2 A.
\end{equation}
\subsection*{Estimate 3b- the relative distances in the inner set are bounded from below.}
This \vc{} is the core of our rigidity Theorem. Together with the estimate 3a it expresses that the particles blow-up with the same rate, homogeneously inside the inner set. 
Equipped with the induction Proposition \ref{induction} we are ready to prove the estimate 3b.
The strategy is to isolate the left-most relative distance with the descent step of Proposition induction: this is the local induction.
Then we exclude the left-most particle with the reinitialization step and repeat the local induction.
Step by step we bound from below every relative distance.
We recall that $\In=[l,l+k-1]$.
\subsection*{Step 1- A lower bound for $|Y_{l+1}-Y_{l}|$: the local induction.}
By Theorem \ref{estimationbornesuperieur}(1), we know that $P^2_{\In}$ is bounded from below by $\frac{1}{A^2}$. On the other hand, we deduce from \eqref{bordsontgrand} that  $|Y_l-Y_{l-1} |\geq \frac{1}{A}$ for $t$ close enough to $T$. These are exactly the conditions for applying the descent step in Proposition \ref{induction}, with $q=l$ and $p=l+k-1$. This yields $\bar B_1>A>0$ such that $P^2_{l,l+k-2}\geq \frac{1}{\bar B_1^2}$. Since $|Y_l-Y_{l-1}| \geq \frac{1}{A} \geq \frac{1}{\bar B_1}$ we can repeatedly apply the same descent step down to $p=l+2$. As a consequence we obtain, in the last iteration, a lower bound, say $\frac{1}{\bar B^2}$, for $P^2_{l,l+1}$. We deduce immediately a lower bound for $|Y_{l+1}-Y_{l}|$:
\begin{align}\label{pourreinitialiser}
\left(  Y_{l+1}-Y_{l} \right)^2 
\geq P^2_{l,l+1}
\geq \frac{1}{\bar B^2}.
\end{align}
\subsection*{Step 2- Not so fast: reinitialization.}
After the first step, it would be natural to exclude the left-most  particle: $Y_l$, and start over the local induction. In fact, this is a delicate issue as we have no information about $P^2_{l+1,l+k-1}$.
This is the reason why the reinitialization step is needed.

For this purpose we use the second information contained in the inequality \eqref{bordsontgrand}, namely: $|Y_{l+k}-Y_{l+k-1} |$ is bounded from below. On the other hand, $\left|  Y_{l+1}-Y_{l} \right|$ is bounded from below also \eqref{pourreinitialiser}.
Therefore the conditions of the reinitialization step in Proposition \ref{induction}  are fulfilled, with $q=l+1$ and $p=l+k-1$. The outcome of the reinitialization step is the required lower bound on $P^2_{l+1,l+k-1}$.

\subsection*{Step 3- Yes we can: The global induction.} We explain here the global induction step.
After the  reinitialization step we can exclude the left-most particle: $Y_l$. By induction on the left-most particle, say $Y_q$, we successively alternate between local induction and reinitialization to exclude $Y_q$, from $q = l$, up to $q=l+k-2$.
In doing so we obtain as a byproduct \eqref{pourreinitialiser} that there exists $B>0$ such that:
$$\forall i\in \In \setminus \left\lbrace l+k-1 \right\rbrace, \quad |Y_{i+1}-Y_i|\geq \frac{1}{B}. $$ This implies the estimate 3b and concludes the proof of Theorem \ref{estimationbornesuperieur}.
\end{proof}

\subsection{Towards a Liouville Theorem}
It is an immediate consequence of Theorem \ref{estimationbornesuperieur} that the rescaled system \eqref{renormalisationN} satisfies the following conditions:
\begin{enumerate}[(R1)]
\item $Y$ is define for all nonnegative time.
\item $\forall i \in \In \quad   Y_i \leq A$.
\item $\forall (i,i+1) \in \In \times \In \quad \left( Y_{i+1}-Y_i\right) \geq \frac{1}{A}.$
\item $\forall j \in  \Out = \left[1,N\right]\setminus \In  \quad | Y_i| \underset{\tau \to+\infty}{\longrightarrow} +\infty$.
\item $\forall \tau \in \R^+ \quad \H_{\In\Out,2}(\tau )\leq A^2 e^{-2\alpha \tau}.$
\end{enumerate}

\begin{Def}[The local rescaled energy functional]\label{energiefonctionnellepartielle}
As usual we fix an inner set of $k$ particles: $\In=[l,l+k-1]$. We define $\E^{\rm resc}_{k}$ by:
$$
\E^{\rm resc}_{k}\left( Y \right)=-\sum_{i \in \In \setminus \lbrace l+k-1 \rbrace  } \log \left(Y_{i+1}-Y_i\right) +\chi h_N \sum_{ (i, j)\in \In \times \In \setminus \lbrace i \rbrace}  \log \vert Y_i-Y_j \vert -\frac{\alpha}{2}  \sum_{i \in \In} \left| Y_i \right|^2.
$$
\end{Def}
This is the rescaled energy restricted to the inner set. Under the rescaled conditions (R1-R5) above, the local energy is bounded from above and below.
We \vc{} have to introduce a technical condition.
We will restrict ourselves to the case where any blow-up set is a strong blow-up set made of $k$ particles.
In this case, according to Theorem \ref{estimationbornesuperieur}, the rescaled solution $Y$ given by \eqref{renormalisationN} satisfies the following condition:
\begin{enumerate}[(R6)]
\item There exists $A>0$ such that for any $i\neq j$, $|Y_i-Y_j|\geq \frac{1}{A}$.
\end{enumerate}

%
We are now ready to give a precise version of  Theorem \ref{Thm:rigidityintro} for the rigidity.
\begin{Thm}\label{kenergydecroit}
Let $Y$ be a solution of the differential equation \eqref{flotgradientdiscretr} satisfying the conditions (R1-R6) then
\begin{itemize}
\item for any $i\in \In$, $\dot Y_i(\tau)\to 0$ as $\tau \to \infty$.
\item  $\E^{\rm resc}_{k}\left(Y(\tau)\right)$ converges to a limit noted $e_\infty$ as $\tau \to \infty$.
\item $\left(\nabla \E^{\rm resc}_{k}\right)\left(Y(\tau)\right) \to 0$ as $\tau \to \infty$.
\end{itemize}

\end{Thm}
Theorem \ref{kenergydecroit} is quite unsatisfactory since it would be natural to expect that $Y(\tau)$ converges (without extracting subsequences) to a critical point of the rescaled energy $\E^{\rm resc}_{k}$. For this purpose it would be interesting to gain more information about the solutions of \eqref{flotgradientrNd} which are defined up to $\tau = +\infty$, in the spirit of the Liouville Theorem in \cite{GK85}. We aim to develop an argument based on the Loyasiewicz inequality, from the theory of gradient flows of analytical energies. However we face technical difficulties and we leave it for future work. Another way to conclude would be to get a better description of the critical points of the functional $\E^{\rm resc}_{k}$. According to the case of three particles in appendix we believe that there is only a finite number of critical points. This would be enough to prove a Liouville Theorem.


Before proving Theorem \ref{kenergydecroit} we remark that a rescaled solution behaves almost like a solution of the local gradient flow.
\begin{Prop}\label{procheflotgradk}
Let $Y$ solution of the differential equation \eqref{flotgradientdiscretr} satisfying the rescaled condition (R1-R5) then there exists $C>0$ such that
\[
\forall \tau>0, \quad \| \nabla \E^{\rm resc}_{k}\left((Y_i)_{i\in \In}\right) -\left( \nabla_i \E^{\rm resc} (Y)\right)_{i\in \In} \|_{l^2(\In)} \leq C e^{-\alpha \tau}.
\]
\end{Prop}
\begin{proof}
From condition (R5) there exists $A$ such that $\H_{\In\Out,2}(\tau)\leq A^2 e^{-2\alpha \tau} $. Then we compute for any $i\in \In=[l,l+k-1]$:
\begin{align*}
\left| \nabla_i \E^{\rm resc}_{k}\left(Y\right) - \nabla_i \E^{\rm resc} (Y) \right| &= \left| -\frac{\delta_{i,l}}{Y_l-Y_{l-1}}+\frac{\delta_{i,l+k}}{Y_{l+k+1}-Y_{l+k}} - 2 \chi h_n  \sum_{k \in \Out } \frac{1}{Y_k-Y_i}  \right|\\
&\leq \left(2 +\frac{2\chi}{\sqrt{N}}\right) A e^{-\alpha \tau}.
\end{align*}
\end{proof}

\begin{proof}[Proof of Theorem \ref{kenergydecroit}]
This proof is divided into two steps.
\subsection*{Step 1-} Under the hypotheses of Theorem \ref{kenergydecroit},
there exists $C>0$ such that,

\[\forall \tau>0, \quad
\|\dot Y_i \|_{l^\infty(\In)} \leq C , \quad
\|Y_j \|_{l^\infty(\Out)}\leq C e^{\alpha \tau} , \quad
\|\dot Y_j \|_{l^\infty(\Out)}\leq C e^{\alpha \tau} , \quad
\|\ddot Y_i \|_{l^\infty(\In)} \leq C .
\]

First, for all $i\in \In$, we have,
\begin{align*}
\left| \dot Y_i \right| &= \left| -\frac{1}{Y_{i+1}-Y_i} + \frac{1}{Y_{i}-Y_{i-1}}+ 2\chi h_N\sum_{k \neq i } \frac{1}{Y_k-Y_i}+ \alpha Y_i \right| \\
&\leq  2A+    2\chi h_N \sum_{k \in \In \setminus \lbrace i \rbrace  } \frac{1}{Y_k-Y_i}  +2\chi h_N\sum_{k \in \Out } \frac{1}{Y_k-Y_i} +\alpha A \\
&\leq \left(2 + \alpha + 2\chi h_N(k-1)\right) A  +A e^{-\alpha \tau}.
\end{align*}
Secondly, for all $j\in \Out$, we have,
\begin{align*}
\left| \dot Y_j - \alpha Y_j \right|  &=  \left| -\frac{1}{Y_{j+1}-Y_j} + \frac{1}{Y_{j}-Y_{j-1}}+ 2\chi h_N\sum_{k \neq j } \frac{1}{Y_k-Y_j}  \right| \\
&\leq   2A +2\chi h_N\sum_{k \in \Out \setminus \lbrace j \rbrace } \frac{1}{|Y_k-Y_j|} +    2\chi h_N \sum_{k \in \In   } \frac{1}{|Y_k-Y_j|}  \\
&\leq \left(2 + \alpha + 2\chi h_N(N-k+1)\right) A  +A e^{-\alpha \tau}.
\end{align*}
Taking $C$ large enough, the Gronwall Lemma yields $\left| Y_j \right| \leq C e^{\alpha \tau}$. By triangular inequality $\left|\dot  Y_j \right| \leq C e^{\alpha \tau}$. \\
Finally, we compute $\ddot Y_{i}$ for $i \in \In$,
\begin{align*}
\left| \ddot Y_{i} \right| &=  \vc{} \left| \frac{d}{d \tau} \left( -\frac{1}{Y_{i+1}-Y_i} + \frac{1}{Y_{i}-Y_{i-1}}+ 2\chi
h_N\sum_{k \neq i } \frac{1}{Y_k-Y_i}+ \alpha Y_i \right) \right|\\
&=\left|  \frac{\dot Y_{i+1}-\dot Y_i}{(Y_{i+1}-Y_i)^2} - \frac{\dot Y_{i}- \dot Y_{i-1}}{(Y_{i}-Y_{i-1})^2}- 2\chi h_N\sum_{k \neq i } \frac{\dot Y_k- \dot Y_i}{(Y_k-Y_i)^2}+ \alpha \dot Y_i     \right|\\
&\leq C_1\left( \left( 4+ \alpha +2\chi h_N(k-1)\right) A^2  +A^2 e^{-2\alpha \tau}  \right)+\\
&\left|\frac{\dot Y_{l+k}}{(Y_{l+k}-Y_{l+k-1})^2}\right| +\left| \frac{\dot Y_{l-1}}{(Y_{l}-Y_{l-1})^2}\right| + 2\chi h_N\sum_{k \neq i } \left|\frac{\dot Y_k}{(Y_k-Y_i)^2} \right|\\
&\leq C +\left( 2+ \frac{2\chi}{\sqrt{N}} \right) C e^{\alpha \tau}A^2 e^{-2\alpha \tau},
\end{align*}

\subsection*{Step 2- .} The time-derivative of $\E^{\rm resc}_{k}\left( Y \right)$ is estimated, using discrete integration by parts, symmetry and Proposition \ref{procheflotgradk}.  
\begin{align}
\nonumber \frac{d}{d \tau}  \E^{\rm resc}_{k} \left( Y(\tau) \right)  &= \left< \left((\dot Y_i)_{i\in \In}\right),\nabla \E^{\rm resc}_{k}\left((Y_i)_{i\in \In}\right)  \right> \\
\nonumber &=\left< \left((\dot Y_i)_{i\in \In}\right),\left( \nabla_i \E^{\rm resc} (Y)\right)_{i\in \In}  \right> +
\label{nesertarien} \left< \left((\dot Y_i)_{i\in \In}\right),\nabla \E^{\rm resc}_{k}\left((Y_i)_{i\in \In}\right) -\left( \nabla_i \E^{\rm resc} (Y)\right)_{i\in \In}  \right> \\
&\leq -\| \dot Y \|_{l^2(\In)}^2 +  \| \dot Y \|_{l^2(\In)}C e^{-\alpha \tau} \\
\label{nesertarien2} & \leq-\| \dot  Y \|^2_{l^2(\In)}\left(1-  C e^{-\alpha \tau} \right)+Ce^{-\alpha \tau}\, ,
\end{align}
where we used the notation $\| \dot Y \|_{l^2(\In)}^2=\sum_{i \in \In } \dot Y_i^2$.
%
We deduce from \eqref{nesertarien2} the integrability of $\|\dot  Y \|^2_{l^2(\In)}$. We choose $\tau_0$ such that $(\forall \tau \geq \tau_0)\;1-Ce^{-\alpha \tau}\geq \frac{1}{2}$. We get  
\begin{align}
\nonumber \int_{\tau_0}^{\infty} \| \dot Y \|^2_{l^2(\In)} &\leq \limsup_{\tau \to +\infty}\left( \int_{\tau_0}^{t} -2\frac{d}{d \tau}  \E^{\rm resc}_{k}\left( Y(\tau) \right) \right)+Ce^{-\alpha \tau}d\tau\\
\label{nesertarien3} &\leq  2(M-m)+\frac{C}{\alpha },
\end{align}
where $M$ and $m$ are respectively the upper and lower bound of $ \E^{\rm resc}_{k}(Y)$, depending only on $N,\chi,A$ by condition (R3). We deduce from the estimates $\|\ddot Y_i\|_{l^\infty(\In)} \leq C $ 
that $\| \dot Y \|_{l^2(\In)} \to 0$ as $\tau \to \infty$.

We eventually prove the convergence of $\E^{\rm resc}_{k}(Y)$. The inequality \ref{nesertarien} implies:
\begin{align*}
&\left| \frac{d}{d \tau}  \E^{\rm resc}_{k} \left( Y(\tau) \right) \right| \leq \frac{3}{2}\| \dot  Y \|^2_{l^2(\In)} + \frac{C^2}{2}e^{-2\alpha \tau}.
\end{align*}
Therefore $\frac{d}{d \tau}  \E^{\rm resc}_{k}(Y)$ is integrable and there exist $e_\infty$ such that $ \E^{\rm resc}_{k}(Y)\to e_\infty$ as $\tau \to +\infty$.
The last estimate is a straightforward consequence of $\dot Y =\left( \nabla \E^{\rm resc} \right)(Y)$ and Proposition \ref{procheflotgradk}. It concludes the proof of Theorem \ref{kenergydecroit}.
\end{proof}

\section{Conclusion and perspectives}\label{sec:conc}
 We prove a rigidity result for the blow-up of the particle scheme \eqref{flotgradientdiscretexplicite}--\eqref{flotgradientdiscretexplicite-boundary}. More precisely, we are able to quantitatively separate the inner and the outer sets of particles. Interestingly, our rigidity result is obtained under the sole condition that the blow-up sets satisfy the weak condition \eqref{weakblowup} and contains the critical number of particles. Under these conditions, we can develop the induction method (Proposition \ref{induction}), then we deduce Theorem \ref{kenergydecroit}. This is indeed the case when the solution belongs to the basins of stability defined in \ref{Thm:stability}. 


This work opens several perspectives. First, it would be interesting to investigate the continuation of system (1.3)-(1.4) after the blow-up time, following \cite{theseannedevys,Dschm09}. Secondly, we could study more general systems, including a nonlinear diffusion, and a power-law interaction kernel.

\begin{appendices}
\section{The case of three particles as a toy problem}
We study  thoroughly the case of three particles. 
 There are  two possible cases  concerning the blow-up occurence:  either  three or two particles collapse. 
It is convenient to introduce the relative distances: $u_1 = X_2 - X_1$ and $u_2 = X_3 - X_2$. The system \eqref{flotgradientdiscretexplicite}--\eqref{flotgradientdiscretexplicite-boundary} becomes:
\begin{equation}\label{flotgradient3u}
\left\{\begin{array}{ccl}
 \dot u_1 =&\dfrac{2}{u_1} -\dfrac{1}{u_2} &- \quad 2\chi h_3\left(
\dfrac{2}{u_1}-\dfrac{1}{u_2} +\dfrac{1}{u_1+u_2}\right)\smallskip\\
\dot u_2 =&\dfrac{2}{u_2} -\dfrac{1}{u_1} &- \quad 2\chi h_3\left(
\dfrac{2}{u_2}-\dfrac{1}{u_1} +\dfrac{1}{u_1+u_2}\right).
\end{array}\right.
\end{equation}
We assume without loss of generality that $u_2\geq u_1$. By symmetry of the system, and uniqueness of the solutions, the diagonal $\{u_2 = u_1\}$ is invariant by the flow.

We recall that the solution to the system \eqref{flotgradient3u} blows-up in finite time when $\chi>\chi_3 = \frac43$. There is a transition at $\chi = \chi_{3}^2 = 2$: for $\chi_3 < \chi < \chi_3^{2}$ two particles cannot collapse, whereas it is possible for $\chi > \chi_3^{2}$.

\subsection{Three particles collapse} \label{sec:3particles}
First, we consider the intermediate case $\chi_3 < \chi < \chi_3^{2}$. In this case, the blow-up set contains three particles.
\begin{Prop}\label{smbu}
Let $T$ be the blow-up time. We have  $u_1\left(t\right),u_2\left(t\right) \rightarrow 0$ as $t \rightarrow T$. Moreover the ratio $\frac {u_2}{u_1}$ is bounded from above and below.
\end{Prop}
\begin{proof}
We show that there exists  $a>0$ such that, if $\frac{u_2}{u_1} \geq a $ then $\frac{u_2}{u_1}$ decreases. Indeed, from \eqref{flotgradient3u}, we get:
\begin{equation*}
\frac{d}{dt} \left(\frac{u_2}{u_1}\right)  =  \frac{1}{u_1^2}\left[ 2\left(\frac{u_1}{u_2} -\frac{u_2}{u_1}\right) \left( 1- \frac{\chi}{\chi_3^{2}} \right)+ \chi_3^{2}\frac{u_2-u_1}{u_2+u_1} \right]\, .
\end{equation*}
Using that $\left( 1- \frac{\chi}{\chi_3^{2}} \right)>0$, we see that $\frac{d}{dt} \left(\frac{u_2}{u_1}\right)<0$ when $ \frac{u_2}{u_1} $ is large enough.
Thus $\frac{u_2}{u_1} $ is bounded  from above, and from below by assumption.
\end{proof}

\subsubsection{Parabolic rescaling}
In the case $\chi_3 < \chi < \chi_3^{2}$ the second moment is linearly decaying, and touches zero exactly at the blow-up time. We rescale the solution in order to fix the second moment to a constant value equal to one, {\em i.e.} we project $X(t)$ on the sphere of radius one. We also rescale time in order to get a solution defined for all time $\tau\geq 0$:
\begin{equation}\label{renormalisation}
Y\left(\tau\left(t\right)\right)= \frac{X\left(t\right) }{R\left(t\right)},
\end{equation}
where $R\left(t\right)=\|X\left(t\right)\|=\sqrt{|X\left(0\right)|^2-2\alpha t}=\sqrt{2\alpha\left(T- t\right)}$, and
$\tau\left(t\right)=-\frac{1}{\alpha}\log\left(\frac{R\left(t\right)}{R(0)}\right)$. Here, $\alpha = 2\left( \frac\chi{\chi_3} - 1\right)$.
%
We define the relative rescaled distances as: $v_1=Y_2-Y_1 $ and $v_2 = Y_3-Y_2$. It satisfies the following system:
\begin{equation}\label{flotgradient3ru}
\left\{\begin{array}{ccl}
 \dot {v_1}  =& \dfrac{2}{v_1}-\dfrac{1}{v_2}&  -\quad 2 \chi h_3
\left(\dfrac{2}{v_1}-\dfrac{1}{v_2}+\dfrac{1}{v_1+v_2} \right) +\alpha v_1\smallskip\\
 \dot {v_2}  =& \dfrac{2}{v_2}-\dfrac{1}{v_1}&  - \quad 2\chi h_3
\left(\dfrac{2}{v_2}-\dfrac{1}{v_1}+\dfrac{1}{v_1+v_2}\right) +\alpha v_2\\
 \end{array}\right.
\end{equation}

Theorem \ref{estimationbornesuperieur} rewrites as follows.
\begin{Prop}\label{estimation3}
In the case $\chi_3 < \chi < \chi_3^{2}$, the solution $(v_1(\tau),v_2(\tau))$ is uniformly bounded from above and below.  
%
\end{Prop}
We shall see that Proposition \ref{estimation3} enables to determine completely the behaviour of the solutions.
\subsubsection{The blow-up profile}
We aim to describe the explosion behaviour. For this purpose we classify the solutions of \eqref{flotgradient3ru} on the sphere $\|Y\| = 1$. A new transition occurs at $\chi = \bar{\chi} = \frac{16}9 \in (\chi_3, \chi_3^{2})$.
\begin{Prop}\label{liouville1}
If $\chi_3< \chi\leq \bar{\chi}$, then there is a unique attractive point for the system  \eqref{flotgradient3ru} restricted to the sphere $\|Y\| = 1$, namely: $(\bar v_1,\bar v_2) = \left(\frac{\sqrt2}{2},\frac{\sqrt2}{2}\right)$.\\
If $\bar{\chi}< \chi< \chi_3^{2}$, there are two symmetric attractive points $\left(\bar v_1(\chi),\bar v_2(\chi)\right)$ and $\left(\bar v_2(\chi),\bar v_1(\chi)\right)$. Moreover we have $\left(\bar v_1(\chi),\bar v_2(\chi)\right) \rightarrow \left(0,\frac{\sqrt{3}}{2}\right)$ when $\chi \rightarrow \chi_3^{2}$.
\end{Prop}

\begin{proof}[Proof of Proposition \ref{liouville1}]
The condition $\|Y\| = 1$ rewrites
\begin{equation}{\label{ligne}}
v_1^2+v_2^2+v_1v_2=\frac32\, .
\end{equation}
We seek stationary points of \eqref{flotgradient3ru} on this curve. Clearly 
$(\bar v_1,\bar v_2) = \left(\frac{\sqrt2}{2},\frac{\sqrt2}{2}\right)$ is one of them. It is attractive if it is
unique.
More generally, the equation of the stationary points of \eqref{flotgradient3ru}  reads:
$$
0=\left(\frac{3}{v_2}-\frac{3}{v_1}\right)\left(1-2\chi h_3\right) +\alpha \left(v_2-v_1\right) .
$$
We assume $v_2 > v_1$ w.l.o.g. We find,
$$
v_1v_2=3\frac{1-2\chi h_3}{\alpha}>0.
$$
In the case $\chi\geq =\bar \chi$, this equations possesses an extra solution, given by
\begin{align*}
\bar v_1(\chi)&=\frac12 \left( \sqrt{\frac32\left(1+2\frac{1-2\chi h_3}{\alpha}\right)} -
\sqrt{\frac32\left(1-6\frac{1-2\chi h_3}{\alpha}\right)}\right)\, .
\\
\bar v_2(\chi)&=\frac12 \left( \sqrt{\frac32\left(1+2\frac{1-2\chi h_3}{\alpha}\right)}
+\sqrt{\frac32\left(1-6\frac{1-2\chi h_3}{\alpha}\right)}\right)\, .
\end{align*}
\end{proof}

\begin{figure}
\includegraphics[width = 0.48\linewidth]{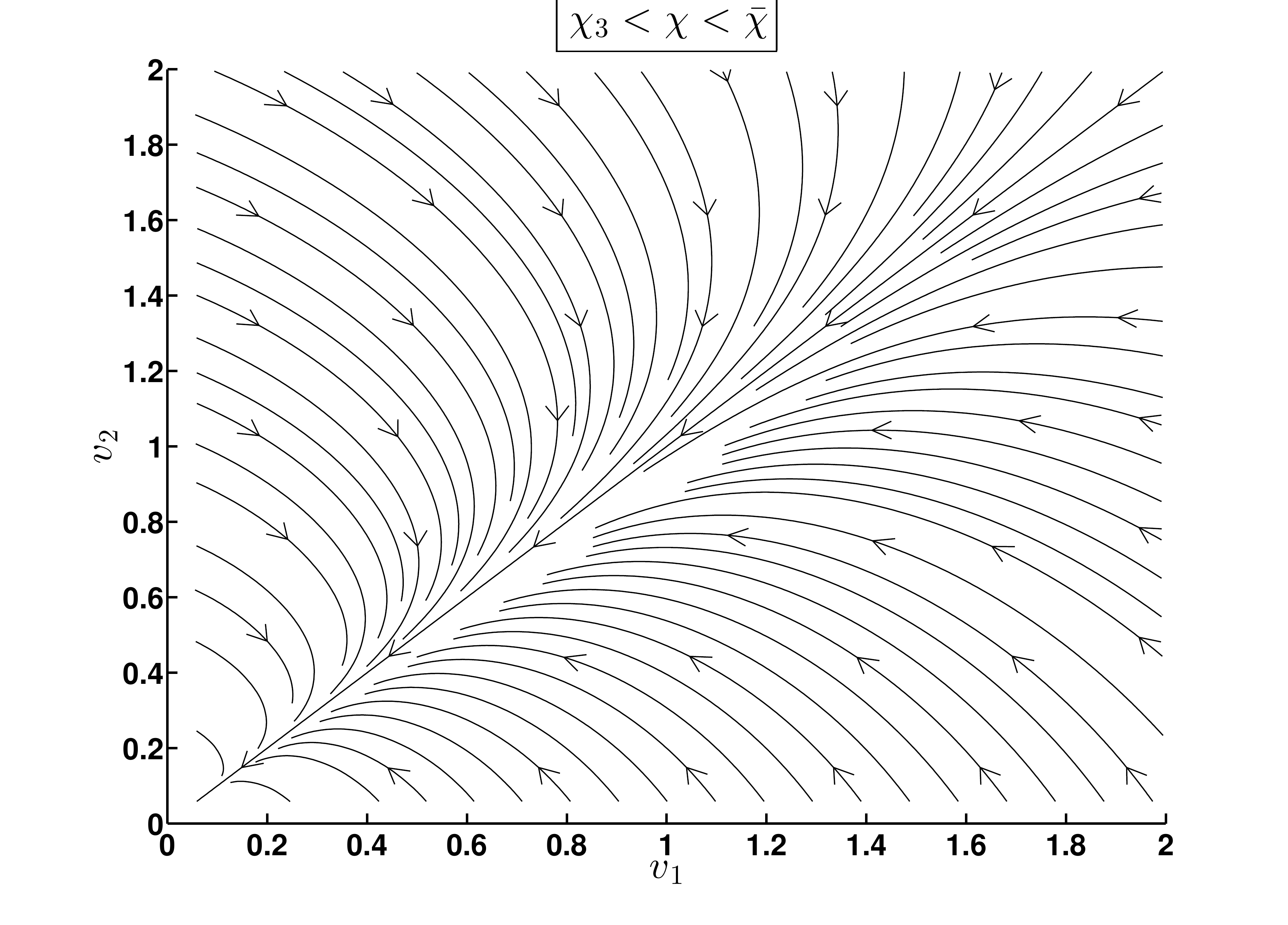}\,
\includegraphics[width = 0.48\linewidth]{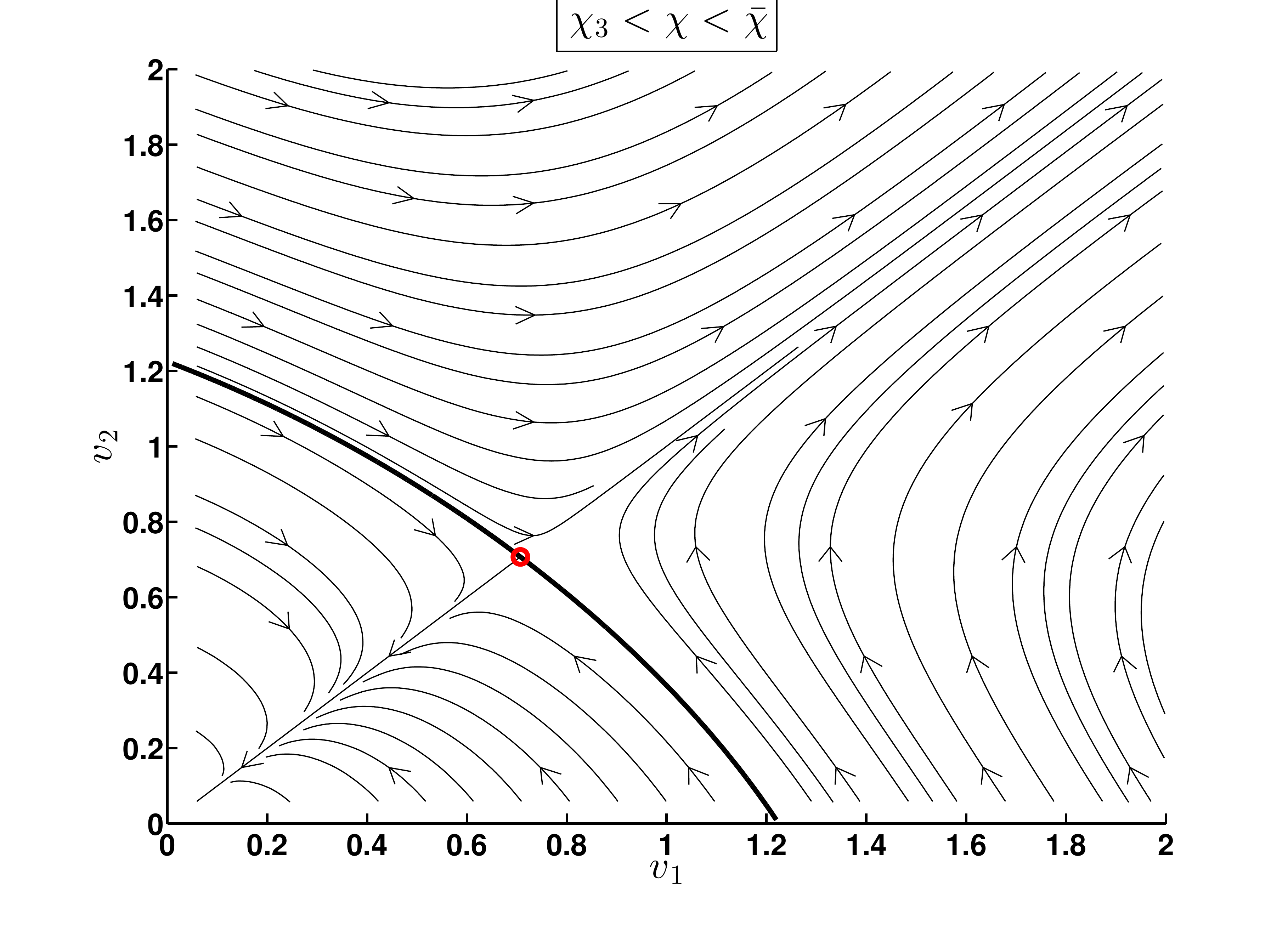}\\
\includegraphics[width = 0.48\linewidth]{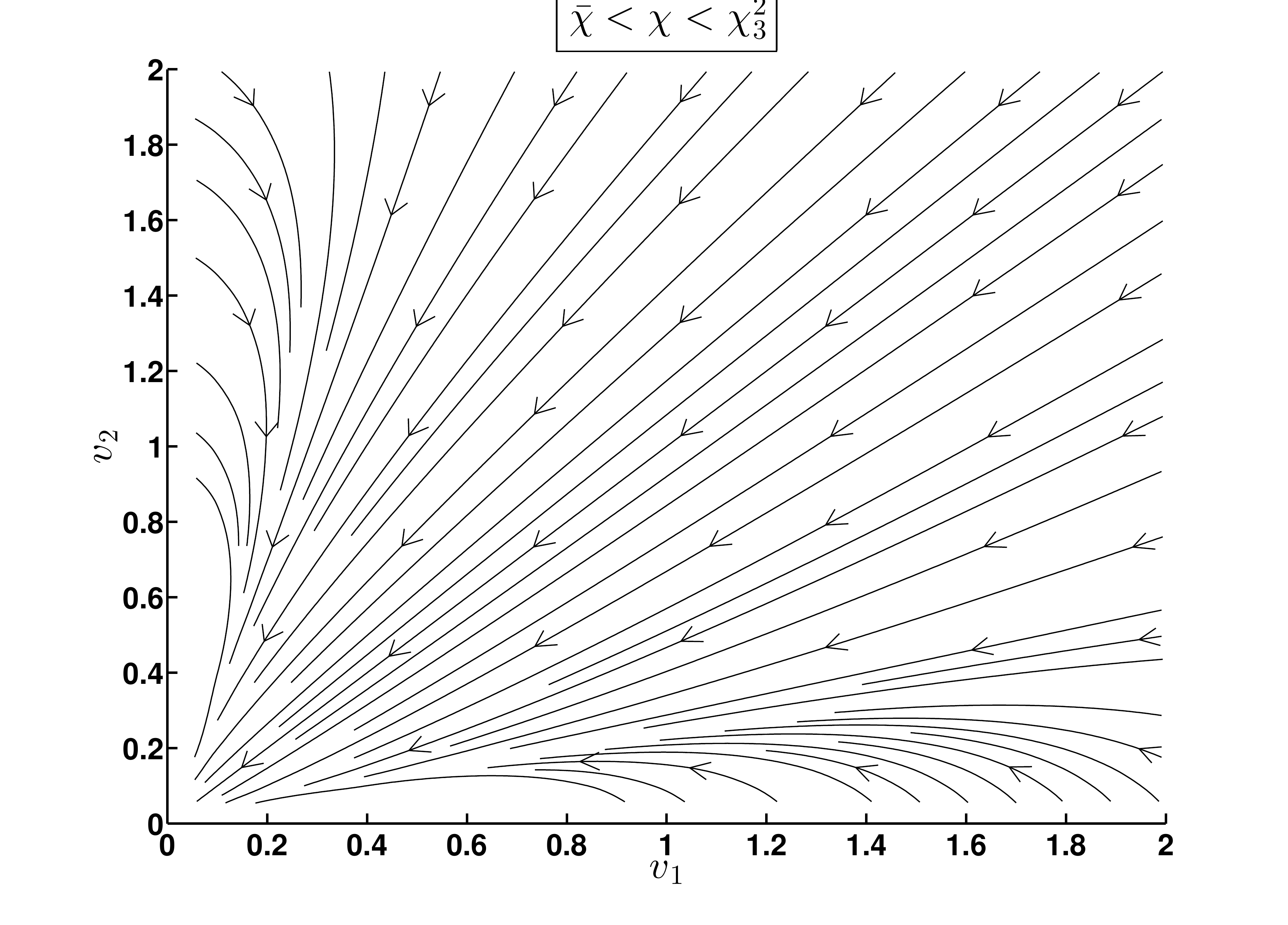}\,
\includegraphics[width = 0.48\linewidth]{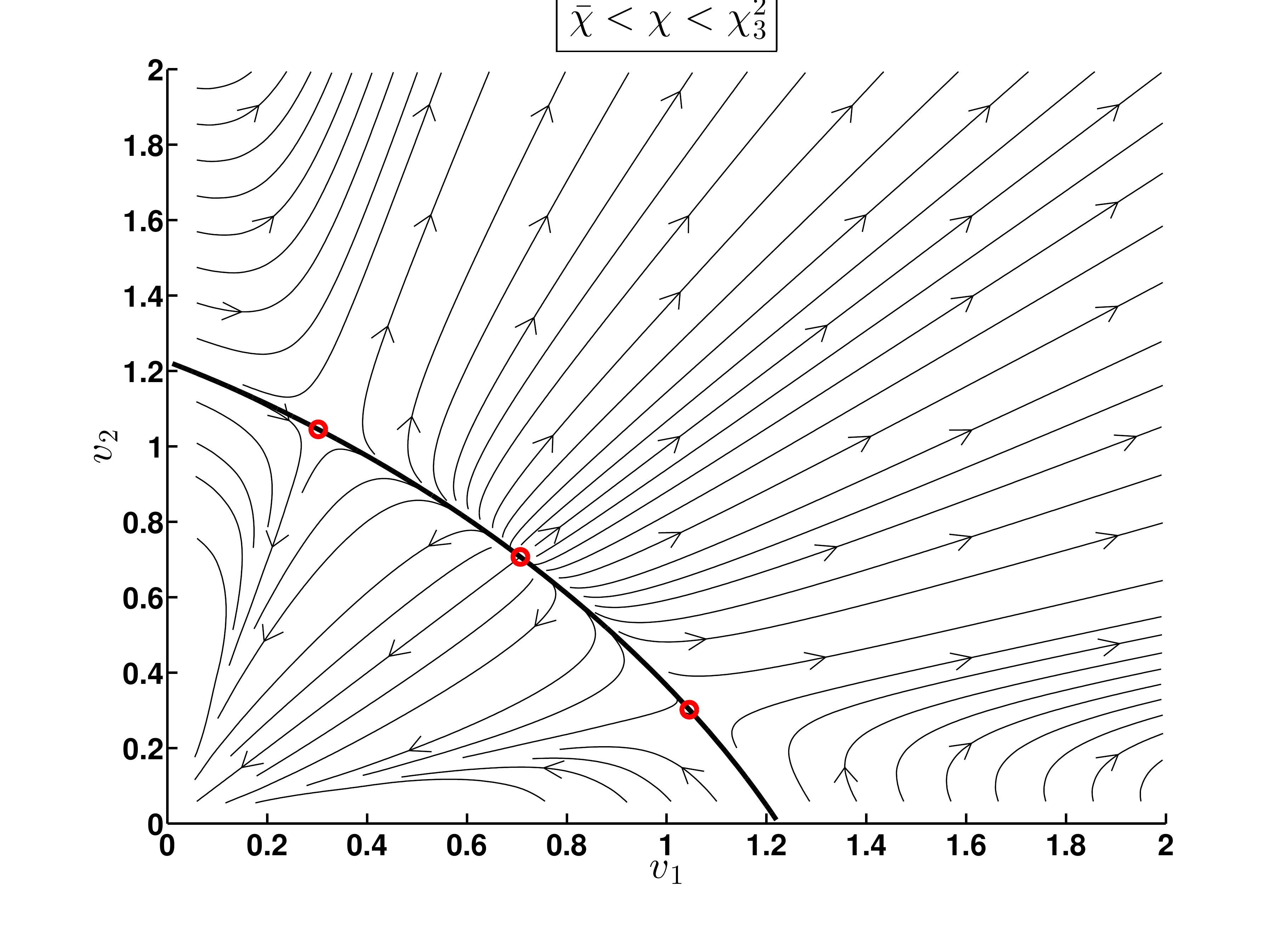}
\caption{Illustration of the dynamics of the three-particles system in the two possible cases: (Top) $\chi_3< \chi< \bar{\chi}$, and (Bottom) $\bar{\chi}< \chi< \chi_3^{2}$. The Left picture shows the dynamics of the original system  \eqref{flotgradient3u}, and the Right picture shows the dynamics of the rescaled system \eqref{flotgradient3ru}. In the former, we clearly see that the three particles collapse simultaneously since the relative distances $v_1$ and $v_2$ both converge to zero. In the latter, the dynamics is restricted to the plain curve defined by \eqref{ligne}. The stationary points are plotted in red circles.}
\label{figsc3p3}
\end{figure}

We are now in position to state a Liouville rigidity theorem for \eqref{flotgradient3u}.
\begin{Thm}[Liouville Theorem]\label{liouville2}
In the case  $\chi_3< \chi\leq \bar{\chi}$, the rescaled solution is  the translation of a unique solution defined for $t\in \R$, except in the trivial symmetric case.
\begin{enumerate}
\item There exists $V$ solution of  $\eqref{flotgradient3ru}$ satisfying $V_1^2+V_2^2+V_1V_2=\frac32$ , defined on $\R$, being such that  $\lim_{\tau \to -\infty }\left(V_1\left(\tau \right),V_2\left( \tau \right)\right)=\left(0,\sqrt{\frac32}\right)$ and $\lim_{\tau \to +\infty }\left(V_1\left(\tau\right),V_2\left(\tau\right)\right)=\left(\frac{\sqrt2}{2},\frac{\sqrt2}{2}\right)$. 
\item Let $\left(v_1,v_2\right)$ be a solution of \eqref{flotgradient3ru}, such that $v_2 > v_1 $, and satisfying
$v_1^2+v_2^2+v_1v_2=\frac32$. Then there exists $s \in \R $ such that:
$$ (\forall \tau>0) \quad v\left(\tau\right)= V \left(\tau+s\right).$$
\end{enumerate}
A similar result holds in the case  $\bar \chi < \chi <\chi_3^2$, but there are two possible branches of solutions. 
\begin{enumerate}
\item There exist two solutions $V^l$ and $V^r$, satisfying $V_1^2+V_2^2+V_1V_2=\frac32$, defined on $\R$, coming respectively from  
$\left(0,\sqrt{\frac32}\right)$ and $\left(\frac{\sqrt2}{2},\frac{\sqrt2}{2}\right) $ as $\tau \to -\infty$, and going both to the attractive point $ \left(\bar v_1(\chi),\bar v_2(\chi)\right)$ as $\tau \to +\infty$. 
\item 
Let $\left(v_1,v_2\right)$ be a solution of \eqref{flotgradient3ru}, such that $v_2 > v_1 $, and satisfying
$v_1^2+v_2^2+v_1v_2=\frac32$. Then there exists $s \in \R $ such that for any $\tau >0$: for any $v=\left(v_1,v_2\right)$, $v_2 \ge v_1 $, solution of $\eqref{flotgradient3ru}$ satisfying
$v_1^2+v_2^2+v_1v_2=\frac32$.
Then there exists $s \in \R $ such that:
$$ (\forall \tau>0) \quad v\left(\tau \right)= V^l \left(\tau+s\right)\, \quad \text{or}\quad (\forall \tau >0) \quad v\left(\tau \right)= V^r \left(\tau+s\right).$$
\end{enumerate}
\end{Thm}

\begin{proof}[Proof of Theorem \ref{liouville2}]

Let  $V=\left(V_1,V_2\right)$ be a maximal solution of \eqref{flotgradient3ru}.
It is defined on $\R$  and satisfies $\lim_{\tau \to-\infty}V(\tau)=\left(0,\sqrt{\frac32}\right)$ and $\lim_{\tau \to+\infty} V(\tau) =\left(\frac{\sqrt2}{2},\frac{\sqrt2}{2}\right)$. Thus $V$ parametrizes the curve \eqref{ligne} above the diagonal: $\{v_2>v_1\}$. 
Consequently for any $v$ solution of $\eqref{flotgradient3ru}$, defined on $[0,+\infty)$ and satisfying
$\eqref{ligne}$, there exists  $s$ such that $v(0)= V(s) $. By uniqueness of the solution, for all $\tau>0$, $v(\tau)= V(\tau+s)$.\\
We can do exactly the same construction in the case $\chi> \bar \chi$.
\end{proof}

We refer to Figure \ref{figsc3p3} for an illustration of these statements.

\begin{Rk}\label{rk:A5}
We can rewrite this theorem with respect to the degrees of freedom of the system. There are two degrees of freedom for the solution of \eqref{flotgradient3u}. After rescaling, the two degrees of freedom are: the blow-up time $T$, and the time shift from the solution $V$:
\begin{equation*}
\left(u_1(t),u_2(t)\right)=\sqrt{2\alpha(T- t)}\left(V_1(\tau(t)+s),V_2(\tau(t)+s)\right).
\end{equation*}
\end{Rk}

\subsubsection{Back to the initial problem}
We make a last important comment: the transition $\chi \lessgtr \bar \chi$ is a first step towards the understanding from the transition from $k+1$ to $k$ particles in the blow-up set as $\chi_N^{k+1}<\chi <\chi_N^k$ increases (here, $k = 2$). Indeed, as $\chi\approx \chi_N^{k+1}$ the blow-up profile is uniquely determined and symmetric. On the other hand, as $\chi\approx \chi_N^{k}$, there are two asymmetric profiles, depending on which particle (here, $X_1$ or $X_3$) contributes the least to the blow-up. As $\chi\to \chi_N^{k}$, the ratio of the asymptotic relative distances diverges, meaning that one of the two extremal particles is progressively ejected from the blow-up set. 

\subsection{Two particles collapse}
Secondly, we assume $\chi > \chi_3^{2}$. Let ($u_1,u_2$) be a solution of \eqref{flotgradient3u}. 
In this case, we expect the following statement (see Figure \ref{figsc2p3}):
\begin{enumerate}
\item If $u_2(0)>u_1(0)$ the blow-up involves $X_1$ and $X_2$ only.
\item If $u_1(0)>u_2(0)$ the blow-up involves $X_2$ and $X_3$ only.
\end{enumerate}
\begin{Rk}
The non generic case $u_1(0)=u_2(0)$ shows that, even if $\chi > \chi_3^{2}$, the blow-up can aggregate three particles, for symmetry reasons.  
\end{Rk}
We suppose without lost of generality that $u_2(0) > u_1(0)$. 
\subsubsection{Parabolic rescaling}
We perform the same parabolic rescaling as in Section \ref{sec:3particles}, except that we substitute $\alpha$ with $\bar{ \alpha } = -2\left(1-\frac{\chi}{\chi^2_{3}}\right)>0$.

Theorem \ref{estimationbornesuperieur} rewrites as follows.
\begin{Prop}\label{ecartcroit}
There exists $A>0$ such that for any $\tau> 0$:
\begin{enumerate}
\item $ \lim_{\tau \to +\infty} v_1(\tau) = 1$.
\item $\frac{1}{A\sqrt{2\bar{\alpha}T}}e^{\alpha \tau}\leq v_2$.
\end{enumerate}
\end{Prop}
\begin{proof}
We have $\liminf_{t\to T^-}u_1=0$.
We begin with the third estimate, namely: $u_2$ is bounded from below. The equation \eqref{flotgradient3u} gives
\begin{align*}
\dot u_2-\dot u_1 =\left(\frac{3}{u_2} -\frac{3}{u_1}\right) \left(1 - 2\chi h_3 \right)=\left(\frac{3}{u_2} -\frac{3}{u_1}\right) \left(1 - \frac{\chi}{\chi_3^{2}} \right).
\end{align*}
Since $u_2(0)> u_1(0)$, and $\left(1 - \frac{\chi}{\chi_3^{2}} \right)\leq 0$,  we deduce that $u_2-u_1$ increases. In particular, for all $t\in[0,T)$:
\begin{equation}
u_2(t)\geq u_2(0)- u_1(0)+u_1(t)\geq  u_2(0)- u_1(0).
\end{equation}
Taking $A\geq \frac{1}{u_2(0)- u_1(0)}$ proves item (ii).

Concerning the first estimate, we start from the non-rescaled equation: 
\begin{align*}
\dot u_1 = \frac{1}{u_1} 2\left(1-2\chi h_3\right) -\frac{1}{u_2} \left(1-2\chi h_3\right) -\frac{2\chi h_3}{u_1+u_2},
\end{align*}
Since $u_1\leq u_2$ we get
\begin{equation}\label{labelpourri}
2 u_1 \dot u_1 = -2 \bar{\alpha } +\frac{u_1}{u_2} \bar{\alpha } -\frac{2 u_1 \chi h_3}{u_1+u_2} \leq - \bar{\alpha }.
\end{equation}
Thus $u_1^2$ decreases, and $ \lim_{ t \rightarrow T} u_1(t) = 0$. Since $u_2$ is bounded from below we deduce that $\frac d{dt} u_1(t)^2 \sim - 2 \bar{\alpha }$. Therefore, $u_1(t)^2 \sim  2 \bar{\alpha} (T-t) $. This concludes the proof of item (i). 
\end{proof}

We finally state a Liouville theorem for the case where two particles only collapse. 

\begin{figure}
\includegraphics[width = 0.48\linewidth]{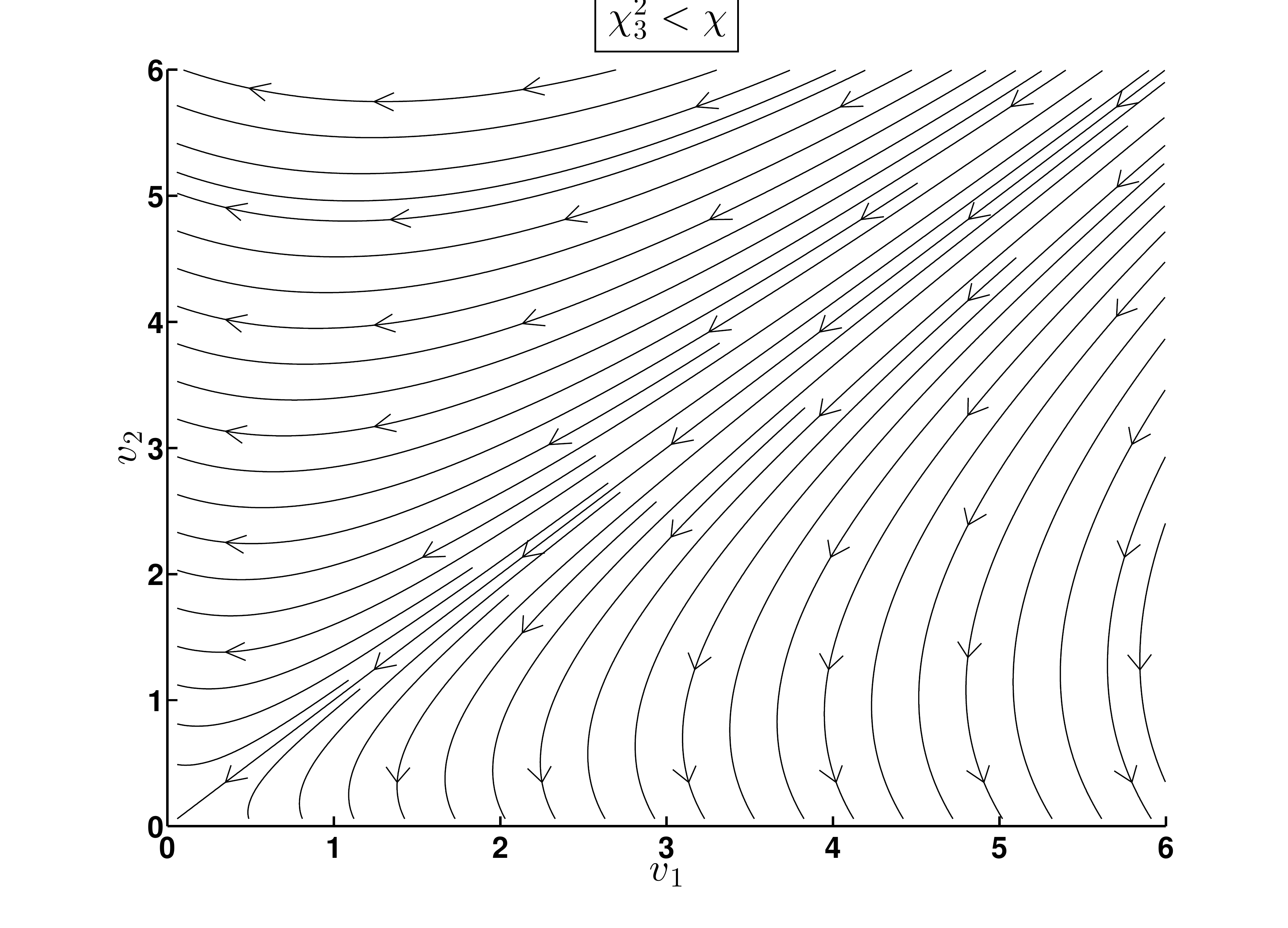}\,
\includegraphics[width = 0.48\linewidth]{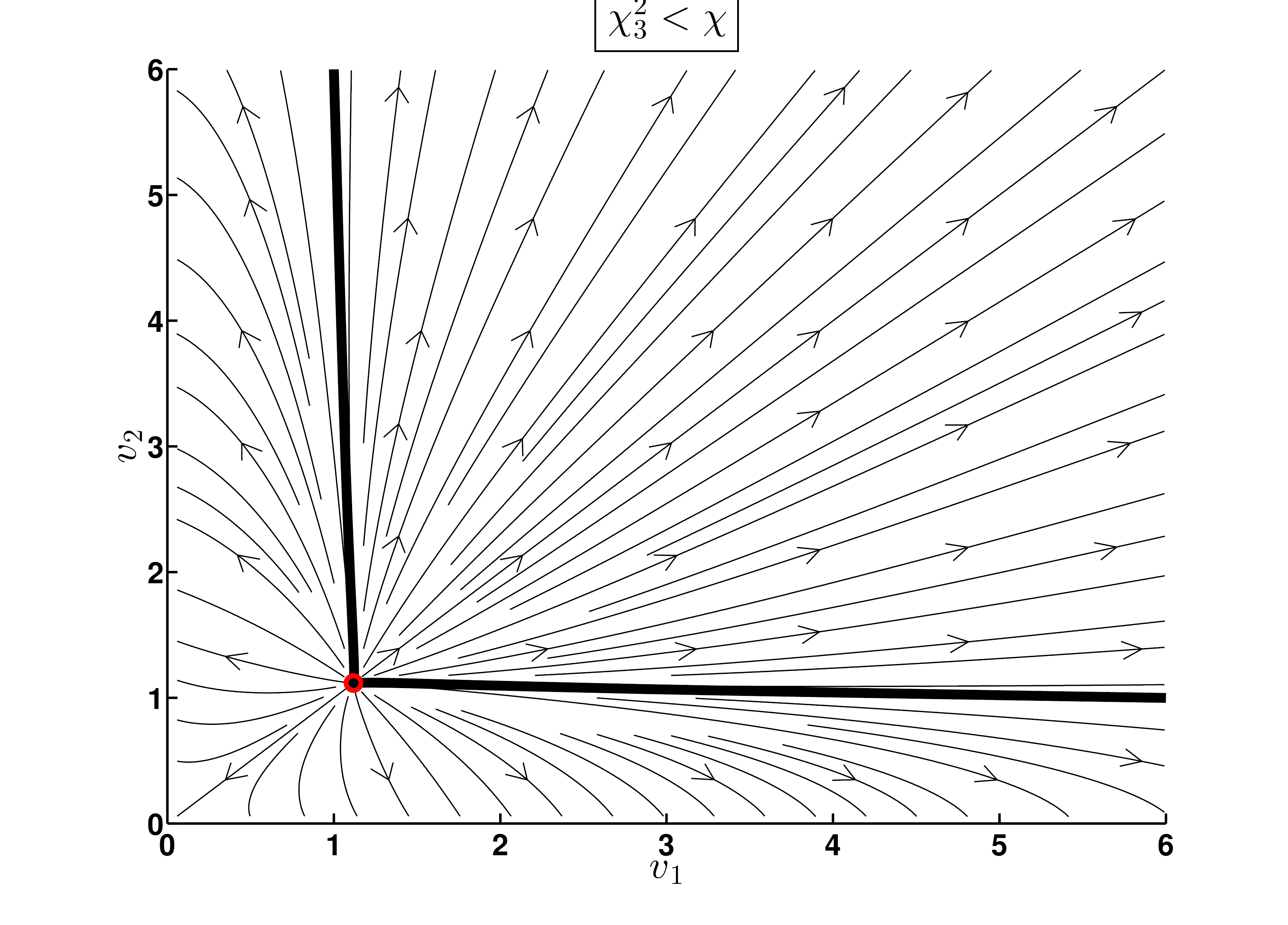}
\caption{Illustration of the dynamics of the three-particles system in the case: $\chi_3^2< \chi$. The Left picture shows the dynamics of the original system \eqref{flotgradient3u}, and the Right picture shows the dynamics of the rescaled system  \eqref{flotgradient3ru}. In the former, we clearly see that only two particles collapse simultaneously, except in the symmetric case $v_1 = v_2$. In the latter, we exhibit the two possible branches of infinite solutions that come  from the diagonal as $\tau\to -\infty$.}
\label{figsc2p3}
\end{figure}

\begin{Thm}[Liouville Theorem]
There exists $\bar V = \left(\bar V_1,\bar V_2\right)$, defined on $\R$, solution of \eqref{flotgradient3ru} such that:
if $v =\left(v_1,v_2\right)$ is a  solution of 
$\eqref{flotgradient3ru}$ defined on $\left[0,+\infty\right)$, satisfying $ v_2 > v_1 $, and verifying the conditions of Proposition \ref{ecartcroit}, then there exists $s \ge 0 $ such that 
$ v \left(\tau \right) = \bar V \left(\tau+s\right)$.
\end{Thm}
\begin{proof}
We perform the change of variables $\left(\xi,\eta \right)=\left( v_1-1,\frac{1}{v_2}\right)$. Linearizing \eqref{flotgradient3ru} near the critical point $ (1,0)$, we get,
$ \left(\dot{\xi},\dot{\eta}\right) = L\left(\xi,\eta\right) + f\left(\xi,\eta\right) $ with
$|f|\le \left(\xi^2+\eta^2\right)$ and 
\[  \begin{pmatrix}
\dot{\xi}\\
\dot{\eta}
\end{pmatrix}  = \begin{pmatrix}
2 \bar \alpha&-1\\
0 & -\bar \alpha
\end{pmatrix}  \begin{pmatrix}
 \xi \\
 \eta 
\end{pmatrix} + O \left(\left\| 
\xi\|^2 + \| 
\eta
\right\|^2\right) \, .\] 
We define $\bar V$ as the stable manifold of the hyperbolic point $(1,+\infty)$. It is 
defined on $\R$ with the boundary condition $\lim_{\tau \to-\infty}V_1=\lim_{\tau \to-\infty}V_2=\sqrt{\frac{3\chi h_3-1}{\bar \alpha}}$.
Since $\lim_{\tau \to \infty}v_1(\tau)=1$ and $\lim_{\tau \to \infty}v_2(\tau)=0$, the solution $v$ lies on the stable manifold $\bar V$. Hence, there exists  $s>0$ such that 
for any $\tau>0$:
$$
v(\tau)= \bar V (\tau+s).
$$ 
 
 \end{proof}
 
We refer to Figure \ref{figsc2p3} for an illustration of these statements.
 
%
\end{appendices}
\bibliographystyle{abbrv}
\bibliography{/Users/gallouet/Dropbox/cotutelle/recherche/these/bibliothese2}
\end{document}